\documentclass[12pt]{amsart}
\usepackage{amsmath,amssymb,amsthm, amscd}

\newcommand\Z{\mathbb{Z}}
\newcommand\Q{\mathbb{Q}}

\def\calO{{\mathcal{O}}}

\def\p{{\mathfrak{p}}}
\def\q{{\mathfrak{q}}}
\def\P{{\mathfrak{P}}}
\DeclareMathOperator{\PGL}{PGL}
\DeclareMathOperator{\GL}{GL}

\DeclareMathOperator{\id}{id}

\newcommand{\invlim}{\mathop{\displaystyle \varprojlim}\limits}

\newenvironment{Gal}{{\rm Gal\,}}{}
\newenvironment{Res}{{\rm Res\,}}{}
\newenvironment{Disc}{{\rm Disc}}{}

\newenvironment{Aut}{{\rm Aut}}{}

\newtheorem{theorem}{Theorem}[section]
\newtheorem{definition}[theorem]{Definition}
\newtheorem{lemma}[theorem]{Lemma}
\newtheorem{proposition}[theorem]{Proposition}
\newtheorem{proposition-definition}[theorem]{Proposition-Definition}
\newtheorem{corollary}[theorem]{Corollary}
\newtheorem{conjecture}[theorem]{Conjecture}

\newtheorem{question}[theorem]{Question} 

\theoremstyle{definition}

\theoremstyle{remark}
\newtheorem*{remark}{Remark}

\begin{document}
\title[Galois theory of quadratic rational functions]{Galois theory of quadratic rational functions}
\author{Rafe Jones}
\author{Michelle Manes}
\thanks{The first author's research was partially supported by NSF grant DMS-0852826, and the second author's by NSF grant DMS-1102858.}

\begin{abstract}
For a number field $K$ with absolute Galois group $G_K$, we consider the action of $G_K$ on the infinite tree of preimages of $\alpha \in K$ under a degree-two rational function $\phi \in K(x)$, with particular attention to the case when $\phi$ commutes with a non-trivial Mobius transfomation.  In a sense this is a dynamical systems analogue to the $\ell$-adic Galois representation attached to an elliptic curve, with particular attention to the CM case.  Using a result about the discriminants of numerators of iterates of $\phi$, we give a criterion for the image of the action to be as large as possible.  This criterion is in terms of the arithmetic of the forward orbits of the two critical points of $\phi$.  In the case where $\phi$ commutes with a non-trivial Mobius transfomation, there is in effect only one critical orbit, and we give  a modified version of our maximality criterion.  We prove a Serre-type finite-index result in many cases of this latter setting.  
\end{abstract}


\maketitle

\section{Introduction} \label{intro}

Let $K$ be a number field, and $\phi \in K(z)$ a rational function of degree $d \geq 2$.  Put 
$$\phi^n = \underbrace{\phi \circ \phi \circ \cdots \circ \phi}_n,$$   
and denote by $\phi^{-n}(\alpha)$ the set of preimages of the point $\alpha$ under the map $\phi^n$.
To the pair $(\phi, \alpha)$, where $\alpha \in K$, we associate a tree of preimages: let $V_n = \phi^{-n}(\alpha)$, and give the set $T_\alpha = \bigcup_{n \geq 1} V_n$ the structure of a tree with root $\alpha$ by assigning edges according to the action of $\phi$.   See Figure \ref{fig: kisone} for examples.  Because elements of $\Gal(\overline{K}/K)$ commute with $\phi$, we obtain a map 
\[
\rho: \Gal(\overline{K}/K) \to \Aut(T_\alpha),
\]
 where $\Aut(T_\alpha)$ denotes the group of tree automorphisms of $T_\alpha$.  We call $\rho$ the {\em arboreal Galois representation} attached to $(\phi, \alpha)$, and the main goal of the present work is to study the image of $\rho$ in the case of degree-two rational functions.

Indeed, given a prime $\ell$ and an elliptic curve $E$ defined over a number field $K$, we obtain the $\ell$-adic Galois representation $\rho_E : \Gal(\overline{K}/K) \to \GL_2(\Z_\ell)$ in much the same manner as the previous paragraph.  The $\ell$-adic Tate module $T_\ell(E)$ is the inverse limit of the sets $[\ell]^{-n}(O)$, and the action of $\Gal(\overline{K}/K)$ on $T_\ell(E)$ gives $\rho_E$.  In this context Serre \cite{serre1} proved that for a given elliptic curve $E$ without complex multiplication the image of $\rho_E$ has finite index in $\GL_2(\Z_\ell)$ for all $\ell$, and $\rho_E$ is surjective for all but finitely many $\ell$.  In the case where $E$ has complex multiplication and the full endomorphism ring is contained in the ground field $K$, similar statements hold (see e.g. \cite[p. 302]{serre1}), provided that $\GL_2(\Z_\ell)$ is replaced by the largest subgroup of $\GL_2(\Z_\ell)$ that commutes with the action on the Tate module induced by the extra endomorphisms of $E$.  In general this is a Cartan subgroup.  

In this paper we formulate a similar conjecture for quadratic rational functions 
and prove it in certain cases.   To do so, we develop a general theory of arboreal representations associated to quadratic rational functions.  In contrast to the situation for $\rho_E$, there appears to be no finite quotient $G$ of the target group such that surjectivity of the the induced representation into $G$ implies surjectivity of $\rho$ (see \cite{itconst} for details).  Rather, infinitely many conditions must be checked, and by studying the ramification of $\rho$ we give a formulation of these in terms of the critical orbits of $\phi$ (Corollary \ref{maxcor2} and Theorem \ref{specmaxcor1}).

When $\phi$ commutes with a non-trivial $f \in \PGL_2(K)$ such that $f(\alpha) = \alpha$, the Galois action on $T_\alpha$ must commute with the action of $f$.  Define the automorphism group of $\phi$ to be 
\begin{align*}
A_{\phi}& = \{ f \in \PGL_2(\overline K) \colon \phi\circ f = f\circ \phi \}, 
\text{ and define } &A_{\phi, \alpha} &= \{f \in A_\phi : f(\alpha) = \alpha\}.
 \end{align*}
 We know that $A_{\phi}$ is finite by~\cite[Proposition 4.65]{jhsdynam}.  
 
   Let $G_\infty$ denote the image of $\rho: \Gal(\overline{K}/K) \to \Aut(T_\alpha)$, and let 
$C_\infty$ denote the centralizer of the action of $A_{\phi, \alpha}$ on $\Aut(T_\alpha)$.    
Recall that $\phi(x) \in K(x)$ is said to be {\em post-critically finite} if the orbit of each of the critical points of $\phi$ is finite.  Such maps have the property that the extension $K(T_\alpha)/K$ is ramified above only finitely many primes of $K$ (see \cite{HC} and the remark following Theorem \ref{discthm}), and thus $G_\infty$ is topologically finitely generated.  One therefore expects it to have infinite index in $C_\infty$, though this is not known in general.  

\begin{conjecture} \label{genmain}
Let $\phi \in K(z)$ have degree $d = 2$ and let $\alpha \in K$.  Suppose that $\phi$ is not post-critically finite.  Then  $[C_{\infty} : G_{\infty}]$ is finite.
\end{conjecture} 

When $A_{\phi,\alpha}$ is trivial, Conjecture~\ref{genmain} has been proven only in the case of two families of quadratic polynomials~\cite[Theorem 1.1 and first remark on p. 534]{quaddiv}, namely 
\[
f(x) = x^2 - kx + k, k \in \Z \quad \text{ and }\quad
f(x) = x^2 + kx - 1, k \in \Z \smallsetminus \{0, 2\},
\]
 for $\alpha = 0$. The key feature of these families is that the orbit of $0$ is finite but not periodic, a property they share with the family in Conjecture~\ref{mainconj} below. We establish here the first similar result for a rational function that is not conjugate to a polynomial. 
\begin{theorem} \label{nonpoly}
Let ${\displaystyle \phi(x) = \frac{1+3x^2}{1-4x-x^2}}$ and $\alpha = 0$. Then $G_\infty \cong \Aut(T_\alpha)$. 
\end{theorem}
The function in Theorem \ref{nonpoly} is polynomial-like in that it has a periodic critical point, though here it is in a 2-cycle rather than being a fixed point. Moreover, the orbit of~$0$ under~$\phi$ is finite but not periodic. See the discussion following Corollary \ref{maxcor2} for a one-parameter family of such maps. 

If $f$ is a non-identity element of $A_{\phi, \alpha}$, then $f(\alpha) = \alpha$, and so $f(\phi^n(\alpha)) = \phi^n(\alpha)$ for all $n \geq 1$.  Because $f$ has exactly two fixed points, $\alpha$ is thus either in a cycle of $\phi$ of length at most two, or maps after one iteration onto a fixed point.  In fact, under the hypotheses of Conjuecture~\ref{genmain}, $\alpha$ is either fixed by $\phi$ or maps to a fixed point of $\phi$ (see Section~\ref{preliminaries}).     Figure~\ref{fig: trees} shows possible pre-image trees under the hypotheses of Conjecture~\ref{genmain}.
Here, 
\begin{equation}
V_1 = \phi^{-1}(\alpha) \setminus \{\alpha\}, \quad V_n = \phi^{-1}(V_{n-1}) \text{ for } n > 1,  \text{ and } \quad T_\alpha = \bigcup_{n \geq 1} V_n.\label{eq : trees}
\end{equation}
Note that $K\left(\phi^{-n}(\alpha)\right) = K(V_n)$.  

\begin{figure}[hbtp] 
\ifx\JPicScale\undefined\def\JPicScale{1}\fi
\unitlength \JPicScale mm
\begin{picture}(71,53.12)(0,0)
\linethickness{0.3mm}
\put(6.86,9.57){\line(0,1){8.57}}
\put(6.86,9.57){\vector(0,-1){0.12}}
\put(6.86,6.71){\makebox(0,0)[cc]{$\phi(\alpha)$}}

\linethickness{0.3mm}
\multiput(8.39,9.7)(0.14,0.2){2}{\line(0,1){0.2}}
\multiput(8.67,10.11)(0.11,0.13){3}{\line(0,1){0.13}}
\multiput(8.99,10.49)(0.12,0.12){3}{\line(0,1){0.12}}
\multiput(9.33,10.83)(0.12,0.1){3}{\line(1,0){0.12}}
\multiput(9.71,11.15)(0.2,0.14){2}{\line(1,0){0.2}}
\multiput(10.1,11.43)(0.21,0.12){2}{\line(1,0){0.21}}
\multiput(10.52,11.66)(0.22,0.1){2}{\line(1,0){0.22}}
\multiput(10.95,11.86)(0.45,0.15){1}{\line(1,0){0.45}}
\multiput(11.4,12.01)(0.46,0.11){1}{\line(1,0){0.46}}
\multiput(11.86,12.13)(0.46,0.07){1}{\line(1,0){0.46}}
\multiput(12.32,12.19)(0.47,0.02){1}{\line(1,0){0.47}}
\multiput(12.79,12.21)(0.46,-0.02){1}{\line(1,0){0.46}}
\multiput(13.25,12.19)(0.46,-0.07){1}{\line(1,0){0.46}}
\multiput(13.71,12.12)(0.45,-0.11){1}{\line(1,0){0.45}}
\multiput(14.16,12)(0.44,-0.16){1}{\line(1,0){0.44}}
\multiput(14.59,11.84)(0.21,-0.1){2}{\line(1,0){0.21}}
\multiput(15.01,11.64)(0.2,-0.12){2}{\line(1,0){0.2}}
\multiput(15.41,11.4)(0.19,-0.14){2}{\line(1,0){0.19}}
\multiput(15.78,11.12)(0.12,-0.11){3}{\line(1,0){0.12}}
\multiput(16.13,10.8)(0.11,-0.12){3}{\line(0,-1){0.12}}
\multiput(16.45,10.45)(0.14,-0.19){2}{\line(0,-1){0.19}}
\multiput(16.74,10.07)(0.13,-0.2){2}{\line(0,-1){0.2}}
\multiput(16.99,9.66)(0.11,-0.22){2}{\line(0,-1){0.22}}
\multiput(17.21,9.23)(0.18,-0.45){1}{\line(0,-1){0.45}}
\multiput(17.38,8.77)(0.14,-0.47){1}{\line(0,-1){0.47}}
\multiput(17.52,8.3)(0.1,-0.48){1}{\line(0,-1){0.48}}
\multiput(17.62,7.82)(0.06,-0.49){1}{\line(0,-1){0.49}}
\multiput(17.67,7.32)(0.01,-0.5){1}{\line(0,-1){0.5}}
\multiput(17.66,6.32)(0.03,0.5){1}{\line(0,1){0.5}}
\multiput(17.58,5.83)(0.07,0.5){1}{\line(0,1){0.5}}
\multiput(17.47,5.34)(0.11,0.49){1}{\line(0,1){0.49}}
\multiput(17.32,4.86)(0.15,0.48){1}{\line(0,1){0.48}}
\multiput(17.13,4.4)(0.1,0.23){2}{\line(0,1){0.23}}
\multiput(16.89,3.95)(0.12,0.22){2}{\line(0,1){0.22}}
\multiput(16.63,3.53)(0.13,0.21){2}{\line(0,1){0.21}}
\multiput(16.33,3.14)(0.15,0.2){2}{\line(0,1){0.2}}
\multiput(16,2.78)(0.11,0.12){3}{\line(0,1){0.12}}
\multiput(15.64,2.44)(0.12,0.11){3}{\line(1,0){0.12}}
\multiput(15.25,2.15)(0.19,0.15){2}{\line(1,0){0.19}}
\multiput(14.85,1.89)(0.2,0.13){2}{\line(1,0){0.2}}
\multiput(14.42,1.67)(0.21,0.11){2}{\line(1,0){0.21}}
\multiput(13.98,1.49)(0.44,0.18){1}{\line(1,0){0.44}}
\multiput(13.53,1.36)(0.45,0.13){1}{\line(1,0){0.45}}
\multiput(13.07,1.27)(0.46,0.09){1}{\line(1,0){0.46}}
\multiput(12.6,1.22)(0.46,0.04){1}{\line(1,0){0.46}}
\multiput(12.14,1.22)(0.46,-0){1}{\line(1,0){0.46}}
\multiput(11.68,1.27)(0.46,-0.05){1}{\line(1,0){0.46}}
\multiput(11.22,1.36)(0.45,-0.09){1}{\line(1,0){0.45}}
\multiput(10.78,1.49)(0.44,-0.14){1}{\line(1,0){0.44}}
\multiput(10.35,1.67)(0.43,-0.18){1}{\line(1,0){0.43}}
\multiput(9.94,1.89)(0.2,-0.11){2}{\line(1,0){0.2}}
\multiput(9.56,2.15)(0.19,-0.13){2}{\line(1,0){0.19}}
\multiput(9.19,2.45)(0.18,-0.15){2}{\line(1,0){0.18}}
\multiput(8.86,2.78)(0.11,-0.11){3}{\line(1,0){0.11}}
\multiput(8.55,3.15)(0.1,-0.12){3}{\line(0,-1){0.12}}
\multiput(8.28,3.55)(0.14,-0.2){2}{\line(0,-1){0.2}}
\multiput(8.05,3.97)(0.12,-0.21){2}{\line(0,-1){0.21}}
\multiput(7.85,4.41)(0.1,-0.22){2}{\line(0,-1){0.22}}
\put(8.39,9.7){\vector(-2,-3){0.12}}

\put(7,21){\makebox(0,0)[cc]{$\alpha$}}

\linethickness{0.3mm}
\put(6.86,23.57){\line(0,1){8.57}}
\put(6.86,23.57){\vector(0,-1){0.12}}
\put(7,35){\makebox(0,0)[cc]{$V_1$}}

\linethickness{0.3mm}
\put(6.86,38.57){\line(0,1){8.57}}
\put(6.86,38.57){\vector(0,-1){0.12}}
\put(7,50){\makebox(0,0)[cc]{$V_2$}}

\put(64,21){\makebox(0,0)[cc]{$\alpha$}}

\linethickness{0.3mm}
\multiput(64,23)(0.12,0.2){50}{\line(0,1){0.2}}
\put(64,23){\vector(-2,-3){0.12}}
\put(71,35){\makebox(0,0)[cc]{$V_1$}}

\linethickness{0.3mm}
\put(70.86,38.57){\line(0,1){8.57}}
\put(70.86,38.57){\vector(0,-1){0.12}}
\put(71,50){\makebox(0,0)[cc]{$V_2$}}

\linethickness{0.3mm}
\multiput(57,33)(0.12,-0.2){50}{\line(0,-1){0.2}}
\put(63,23){\vector(2,-3){0.12}}
\put(56,36){\makebox(0,0)[cc]{$\alpha$}}

\linethickness{0.3mm}
\multiput(56,38)(0.12,0.2){50}{\line(0,1){0.2}}
\put(56,38){\vector(-2,-3){0.12}}
\put(63,50){\makebox(0,0)[cc]{$V_1$}}

\linethickness{0.3mm}
\multiput(49,48)(0.12,-0.2){50}{\line(0,-1){0.2}}
\put(55,38){\vector(2,-3){0.12}}
\put(49,51){\makebox(0,0)[cc]{$\alpha$}}

\linethickness{0.3mm}
\multiput(66.66,19.27)(0.14,-0.16){2}{\line(0,-1){0.16}}
\multiput(66.94,18.95)(0.13,-0.18){2}{\line(0,-1){0.18}}
\multiput(67.2,18.6)(0.11,-0.19){2}{\line(0,-1){0.19}}
\multiput(67.43,18.22)(0.1,-0.21){2}{\line(0,-1){0.21}}
\multiput(67.63,17.8)(0.17,-0.44){1}{\line(0,-1){0.44}}
\multiput(67.8,17.36)(0.14,-0.46){1}{\line(0,-1){0.46}}
\multiput(67.93,16.91)(0.1,-0.48){1}{\line(0,-1){0.48}}
\multiput(68.03,16.43)(0.07,-0.49){1}{\line(0,-1){0.49}}
\multiput(68.1,15.94)(0.03,-0.5){1}{\line(0,-1){0.5}}
\multiput(68.12,14.95)(0.01,0.5){1}{\line(0,1){0.5}}
\multiput(68.08,14.45)(0.04,0.5){1}{\line(0,1){0.5}}
\multiput(68,13.96)(0.08,0.49){1}{\line(0,1){0.49}}
\multiput(67.89,13.48)(0.11,0.48){1}{\line(0,1){0.48}}
\multiput(67.74,13.02)(0.15,0.46){1}{\line(0,1){0.46}}
\multiput(67.56,12.58)(0.09,0.22){2}{\line(0,1){0.22}}
\multiput(67.35,12.16)(0.11,0.21){2}{\line(0,1){0.21}}
\multiput(67.11,11.77)(0.12,0.19){2}{\line(0,1){0.19}}
\multiput(66.84,11.41)(0.13,0.18){2}{\line(0,1){0.18}}
\multiput(66.55,11.09)(0.15,0.16){2}{\line(0,1){0.16}}
\multiput(66.24,10.8)(0.16,0.14){2}{\line(1,0){0.16}}
\multiput(65.9,10.56)(0.17,0.12){2}{\line(1,0){0.17}}
\multiput(65.56,10.36)(0.17,0.1){2}{\line(1,0){0.17}}
\multiput(65.2,10.21)(0.36,0.15){1}{\line(1,0){0.36}}
\multiput(64.84,10.11)(0.37,0.11){1}{\line(1,0){0.37}}
\multiput(64.46,10.05)(0.37,0.06){1}{\line(1,0){0.37}}
\multiput(64.09,10.04)(0.37,0.01){1}{\line(1,0){0.37}}
\multiput(63.72,10.08)(0.37,-0.04){1}{\line(1,0){0.37}}
\multiput(63.36,10.17)(0.36,-0.09){1}{\line(1,0){0.36}}
\multiput(63,10.3)(0.36,-0.13){1}{\line(1,0){0.36}}
\multiput(62.66,10.48)(0.17,-0.09){2}{\line(1,0){0.17}}
\multiput(62.33,10.71)(0.16,-0.11){2}{\line(1,0){0.16}}
\multiput(62.02,10.97)(0.15,-0.13){2}{\line(1,0){0.15}}
\multiput(61.74,11.28)(0.14,-0.15){2}{\line(0,-1){0.15}}
\multiput(61.48,11.63)(0.13,-0.17){2}{\line(0,-1){0.17}}
\multiput(61.24,12)(0.12,-0.19){2}{\line(0,-1){0.19}}
\multiput(61.03,12.41)(0.1,-0.2){2}{\line(0,-1){0.2}}
\multiput(60.86,12.85)(0.17,-0.43){1}{\line(0,-1){0.43}}
\multiput(60.72,13.3)(0.14,-0.46){1}{\line(0,-1){0.46}}
\multiput(60.61,13.77)(0.11,-0.47){1}{\line(0,-1){0.47}}
\multiput(60.54,14.26)(0.07,-0.49){1}{\line(0,-1){0.49}}
\multiput(60.5,14.76)(0.04,-0.49){1}{\line(0,-1){0.49}}
\put(60.5,14.76){\line(0,1){0.5}}
\multiput(60.5,15.25)(0.04,0.5){1}{\line(0,1){0.5}}
\multiput(60.54,15.75)(0.07,0.49){1}{\line(0,1){0.49}}
\multiput(60.61,16.24)(0.11,0.48){1}{\line(0,1){0.48}}
\multiput(60.72,16.72)(0.14,0.47){1}{\line(0,1){0.47}}
\multiput(60.86,17.19)(0.17,0.45){1}{\line(0,1){0.45}}
\multiput(61.03,17.64)(0.1,0.21){2}{\line(0,1){0.21}}
\multiput(61.24,18.06)(0.12,0.2){2}{\line(0,1){0.2}}
\put(66.66,19.27){\vector(-1,1){0.12}}

\end{picture}
\vspace{-0.1 in}
\caption{The first few levels of typical preimage trees when $\#A_{\phi, \alpha} > 1$.} \label{fig: trees}
\end{figure}
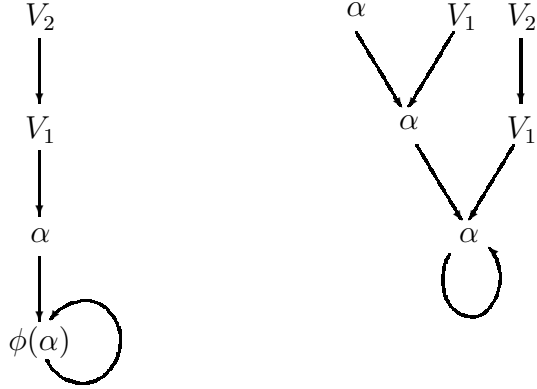

In Section \ref{galgen}, we give a criterion for a given pair $(\phi, \alpha)$ to satisfy $[C_\infty : G_\infty] < \infty$ in the case $d = 2$.  In Section \ref{preliminaries} we show that in the case $\#A_{\phi, \alpha} > 1$, Conjecture~\ref{genmain} is equivalent to:  

\begin{conjecture} \label{mainconj}
Let $\phi(x) = k(x^2 + 1)/x$ and $\alpha = 0$, and suppose that $\phi$ is not post-critically finite.  Then 
$[C_\infty : G_\infty] < \infty$.  
\end{conjecture}

The conjectures above suggest a more general question.
\begin{question} \label{questmain}
Let $\phi \in K(z)$ have degree $d \geq 2$ and $\alpha \in K$.  Under what conditions is $[C_{\infty} : G_{\infty}]$  finite?
\end{question}

The post-critically finite maps of the form $\phi(x) = k(x^2 + 1)/x$ must have $k$ with height at most $2$, and $k$ can only be divisible by primes of $K$ lying over $(2)$ (see Proposition \ref{pcf} for details).  
The structure of $C_\infty$ in the setting of Conjecture~\ref{mainconj} is described at the beginning of Section \ref{galaut}.  In particular, $C_\infty$ is an infinite-index subgroup of $\Aut(T_{\alpha})$ with Hausdorff dimension 1/2 (see p.~\pageref{hausdorff} for the definition);  however, in contrast to Cartan subgroups of $\GL_2(\Z_\ell)$, it is highly non-abelian.  Indeed, $C_\infty$ has an index-two subgroup isomorphic to $\Aut(T_\alpha)$.

One of our main results is the following.  
\begin{theorem} \label{odddeg}
If $K$ is a number field of odd degree over $\Q$, then Conjecture $\ref{mainconj}$ is true for all $k$ in a congruence class, and thus for a positive density of $k$.   
\end{theorem}
 
 We prove that Conjecture \ref{mainconj} is true in many other circumstances (see Corollary \ref{irredfam}).  For simplicity, we state here the result for the case $K = \Q$. 

\begin{theorem} \label{mainthm}
Conjecture~$\ref{mainconj}$ is true for $K = \Q$ provided $k$ satisfies one of the following conditions \textup(we write $v_p$ for the $p$-adic valuation\textup):
\begin{itemize}
\item $v_2(k) = 0$ or $v_3(k) = 0$,
\item $k \equiv 2,3 \bmod{5}$ or $k \equiv 1,2,5,6 \bmod{7}$,
\item $v_p(2k \pm 1) > 0$ for some $p \equiv 3, 5 \bmod{8}$,
\item $v_p(2k^2 -k + 1) > 0$ for some prime $p$ with $-k$ not a square mod $p$, or
\item $v_p(2k^2 + k + 1) > 0$ for some prime $p$ with $k$ not a square mod $p$.
\end{itemize}
\end{theorem}

In the case where $k \in \Z$, we use Theorem \ref{mainthm} plus other results to verify Conjecture~\ref{mainconj} for all $k$ with $|k| \leq 10000$ (see the remark following Corollary \ref{irredfam}).  
We also give the following sufficient conditions on $k$ to ensure that the index in Conjecture~\ref{mainconj} is one (see Theorem~\ref{maxthm}).  
\begin{theorem} \label{mainmax}
Let $\phi$ and $\alpha$ be as in Conjecture~$\ref{mainconj}$, and suppose that $K = \Q$.  There exists an effectively computable set $\Sigma$ of primes of $\Z$ of natural density zero, such that if $v_p(k) = 0$ for all primes belonging to $\Sigma$ then $G_\infty \cong C_\infty$. 
\end{theorem}

For more on $\Sigma$, see Corollary \ref{maincorQ} and the remark following.  We note that $\Sigma$ contains~$2$ and primes congruent to 1 modulo 4, so that if $k$ is an integer divisible only by primes congruent to 3 modulo 4, then Theorem \ref{mainmax} applies.  We note that Theorem \ref{mainmax} may be far from best possible; indeed, we have been unable to find a single $k \in \Z$ for which $[G_\infty : C_\infty] > 1$.  The fact that $\Sigma$ has zero density is a consequence of our analysis of $C_\infty$ and a result relating the structure of $G_\infty$ to the density of prime divisors of orbits of a large class of rational fuctions (Theorem \ref{upbound}).

In order to prove our main results, we generalize techniques from \cite{quaddiv, odoni, stoll} that treat the case of $\phi$ a polynomial.  In particular, in Theorem~\ref{discthm} we obtain a formula for the discriminant of the numerator of $\phi^n$, where $\phi$ is a rational function of degree $d \geq 2$, a problem that has interest in its own right (see~\cite{HC}).   In the case $d=2$, we examine the irreducibility of the numerators of $\phi^n$, both in the general quadratic case (Theorem~\ref{irredcritgen}) and in the case $\phi(x) = k(x^2 + b)/x$ (Theorem \ref{irredthm}).  
We also analyze the extensions $K_n/K_{n-1}$, where $K_i = K(\phi^{-i}(\alpha))$.  We give a criterion for $[K_{n} : K_{n-1}]$ to be as large as possible, both in the general quadratic case (Corollary \ref{maxcor2}) and in the case where $\phi(x) = k(x^2 + b)/x$ (Theorem \ref{specmaxcor1}).  We use the former criterion to prove Theorem \ref{nonpoly} (see the discussion following Corollary \ref{maxcor2}). The criteria for both irreducibility and maximality of the field extensions are arithmetic, and depend on knowledge of primes dividing elements of the form $\phi^n(\gamma)$, where $\gamma$ is a critical point of $\phi$.  We assemble these pieces to prove the following result, which is the main engine behind 
Theorems \ref{odddeg}, \ref{mainthm}, and \ref{mainmax}.  
\begin{theorem} \label{finiteindex1}
Let $\phi(x) = k(x^2 + 1)/x$ and $\alpha = 0$, and suppose that $\phi$ is not post-critically finite.  Put 
\[
\delta_n  = 
\begin{cases}
2k^2 & \text{ if } n=1,\\
\delta_{n-1}^2 + \epsilon_{n-1}^2 & \text{ if } n\geq 2, 
\end{cases}
\qquad \text{and} \qquad
\epsilon_n =
\begin{cases}
 k & \text{ if } n=1,\\
\frac{\delta_{n-1}\epsilon_{n-1}}{k}  &  \text{ if } n \geq 2.
\end{cases}
 \]  
 If none of $-1$, $\delta_n$, or $-\delta_n$ is a square for $n \geq 2$, then $[G_\infty : C_\infty]$ is finite.  
\end{theorem}

See Theorem \ref{finiteindex} for a slightly more general statement.  The sequence $(\delta_n, \epsilon_n)$ is related to the orbit of the critical point $1$ of $\phi$ in that $\phi^n(1) = \delta_n/\epsilon_n$.  What makes Theorem \ref{finiteindex1} possible is the fact that $0$ is a pre-periodic point for $\phi$, which ensures that the set of common prime ideal divisors of $\delta_i$ and $\delta_j$ is very restricted (see Lemma \ref{relprime}).   This allows one to show that except in very special circumstances there must be a primitive prime divisor $\p$ of $\delta_n$ (that is, $\p$ does not divide $\delta_i$ for $i < n$) that divides $\delta_n$ to odd multiplicity.  This is the key hypothesis of Theorem \ref{specmaxcor1}.  The condition that $0$ is pre-periodic has also been used to study primitive prime divisors in other dynamical sequences \cite{xander, rice}.  A natural hope is that similar techniques might be used to tackle Conjecture~\ref{genmain} and even Question~\ref{questmain}, even in the case where $A_{\phi}$ is trivial.

\section{Preliminaries and notation} \label{preliminaries}

In this section we fix some notation, and we show how to reduce  Conjecture~\ref{genmain} to Conjecture~\ref{mainconj} when   $\# A_{\phi,\alpha} > 1$.  For any $g \in \PGL_2(\overline K)$,  define the conjugate map
\[
\phi^g = g\phi g^{-1} .
\]

\begin{proposition}\label{prop: twist prop}
Suppose $\phi \in K(z)$ and basepoint $\alpha \in K$ satisfy $[C_\infty : G_\infty] < \infty $.  Let   $g \in \PGL_2(\overline K)$ such that  $\phi^g  \in K(z)$.  Then  the finite index result  also holds for $\phi^g$ with the basepoint $g(\alpha)$.
\end{proposition}

\begin{proof}
To simplify notation, we let $\psi = \phi^g$.  A  computation reveals that  
\begin{equation}
A_{\psi, g(\alpha)} = g \circ A_{\phi, \alpha} \circ g^{-1}.  \label{eqn: conj aut groups}
\end{equation}
First suppose $g \in \PGL_2(K)$.  Then since
\begin{equation} \label{conjugacy}
\psi^{-n}(g(\alpha)) = \{g(\beta) : \beta \in \phi^{-n}(\alpha)\},
\end{equation}
 $\psi^{-n}(g(\alpha))$ and $\phi^{-n}(\alpha)$ generate the same extension of $K$.  
  
Now suppose that $\psi$ is a nontrivial twist of $\phi$, meaning $g \in \PGL_2(L)$ for some finite extension $L/K$ (here we take $L$ minimal).  From~\cite[Lemma~2.6]{lmt_explicit}, there is an absolute bound $B$ depending only on $\phi$ so that $[L:K] \leq B$.  (In fact, this bound $B$ can be chosen to depend only on the degree of~$\phi$ and not on the specific map.)

By equation~\eqref{conjugacy}, for each $n \geq 1$ an extension of $\phi^{-n}(\alpha)$ of degree at most $B$ contains $K\left(\psi^{-n}(g(\alpha))\right)$.  The finite index result  follows.
\end{proof}

  When $\deg \phi=2$,
$A_\phi$ is either trivial, cyclic of order two, or isomorphic to $S_3$ \cite{milnor}.  The third option occurs if and only if $\phi$ is conjugate over $\overline{K}$ to $1/z^2$, in which case $\phi$ is post-critically finite.  
Hence, if we assume $\# A_{\phi,\alpha} > 1$ in Conjecture~\ref{genmain}, then necessarily $\#A_{\phi, \alpha} = 2$ and $A_{\phi,\alpha} = A_\phi$.  
In this case we have from~\cite[Lemma 1]{mmplms} that $\phi$ is conjugate over $\overline{K}$ to 
\[
\psi(z) = k(z^2 + 1)/z, \  \text{ with } k \in K^* \smallsetminus \{0,  -1/2\}.
\]

  Let $g$ denote the conjugacy such that $\phi^g$ has the form above.  By equation~\eqref{eqn: conj aut groups},  we conclude that $\#A_{\psi, g(\alpha)} = 2$.   Therefore $A_{\psi, g(\alpha)} =  \{\text{id}, z \mapsto -z\}$, and since $g(\alpha)$ must be fixed by elements of this set, it follows that $g(\alpha) = 0$ or $g(\alpha) = \infty$.  Hence $g(\alpha)$ is either fixed by $\psi$ or maps to a fixed point, so   $\alpha$ is either fixed by $\phi$ or maps to a fixed point of $\phi$, as shown in Figure~\ref{fig: trees}. These two cases are illustrated in the specific case $k =  1$ in Figure \ref{fig: kisone}.

\begin{figure}[htbp]
\begin{center} 
\ifx\JPicScale\undefined\def\JPicScale{1}\fi
\unitlength \JPicScale mm
\begin{picture}(182,62.31)(0,0)
\put(7,62.31){\makebox(0,0)[tc]{$-\sqrt{\frac{-3 - \sqrt{5}}{2}}$}}

\put(27.8,62.31){\makebox(0,0)[tc]{$\sqrt{\frac{-3 + \sqrt{5}}{2}}$}}

\put(55.8,62.31){\makebox(0,0)[tc]{$-\sqrt{\frac{-3 + \sqrt{5}}{2}}$}}

\put(78.2,62.31){\makebox(0,0)[tc]{$\sqrt{\frac{-3 - \sqrt{5}}{2}}$}}

\linethickness{0.36mm}
\multiput(67.05,36.33)(0.12,0.24){66}{\line(0,1){0.24}}
\put(67.05,36.33){\vector(-1,-2){0.12}}
\linethickness{0.36mm}
\multiput(58.58,52.49)(0.12,-0.28){58}{\line(0,-1){0.28}}
\put(65.56,36.33){\vector(1,-2){0.12}}
\linethickness{0.36mm}
\multiput(19.65,36.98)(0.12,0.31){50}{\line(0,1){0.31}}
\put(19.65,36.98){\vector(-1,-3){0.12}}
\linethickness{0.36mm}
\multiput(10.18,52.49)(0.12,-0.25){62}{\line(0,-1){0.25}}
\put(17.66,36.98){\vector(1,-2){0.12}}
\put(65.98,34.66){\makebox(0,0)[tc]{$i$}}

\put(18.32,34.66){\makebox(0,0)[tc]{$-i$}}

\linethickness{0.36mm}
\multiput(44.2,4.65)(0.12,0.15){170}{\line(0,1){0.15}}
\put(44.2,4.65){\vector(-3,-4){0.12}}
\linethickness{0.36mm}
\multiput(20.3,29.51)(0.12,-0.14){173}{\line(0,-1){0.14}}
\put(41.5,4.65){\vector(3,-4){0.12}}
\put(42.75,3.37){\makebox(0,0)[tc]{$0$}}

\put(98,61){\makebox(0,0)[tc]{$\infty$}}

\put(116.3,61.31){\makebox(0,0)[tc]{$0$}}

\put(132.75,61.31){\makebox(0,0)[tc]{$-i$}}

\put(152.56,61.31){\makebox(0,0)[tc]{$i$}}

\linethickness{0.36mm}
\multiput(142.71,35.33)(0.12,0.26){81}{\line(0,1){0.26}}
\put(142.71,35.33){\vector(-1,-2){0.12}}
\linethickness{0.36mm}
\multiput(133.82,56)(0.12,-0.33){63}{\line(0,-1){0.33}}
\put(141.38,35.33){\vector(1,-3){0.12}}
\linethickness{0.36mm}
\multiput(107.86,35.98)(0.12,0.32){62}{\line(0,1){0.32}}
\put(107.86,35.98){\vector(-1,-3){0.12}}
\linethickness{0.36mm}
\multiput(98.44,56)(0.12,-0.31){64}{\line(0,-1){0.31}}
\put(106.1,35.98){\vector(1,-3){0.12}}
\put(141.77,33){\makebox(0,0)[tc]{$0$}}

\put(106.68,33){\makebox(0,0)[tc]{$\infty$}}

\linethickness{0.36mm}
\multiput(126.74,4)(0.12,0.2){118}{\line(0,1){0.2}}
\put(126.74,4){\vector(-2,-3){0.12}}
\linethickness{0.36mm}
\multiput(108.43,28.51)(0.12,-0.2){123}{\line(0,-1){0.2}}
\put(123.2,4){\vector(2,-3){0.12}}
\put(125.09,3){\makebox(0,0)[tc]{$\infty$}}

\put(161.08,44){\makebox(0,0)[cc]{}}

\end{picture}

\end{center}
\caption{First few levels of the preimage trees of $0$ and $\infty$ under $\phi(x) = (x^2 + 1)/x$}  \label{fig: kisone}
\end{figure}
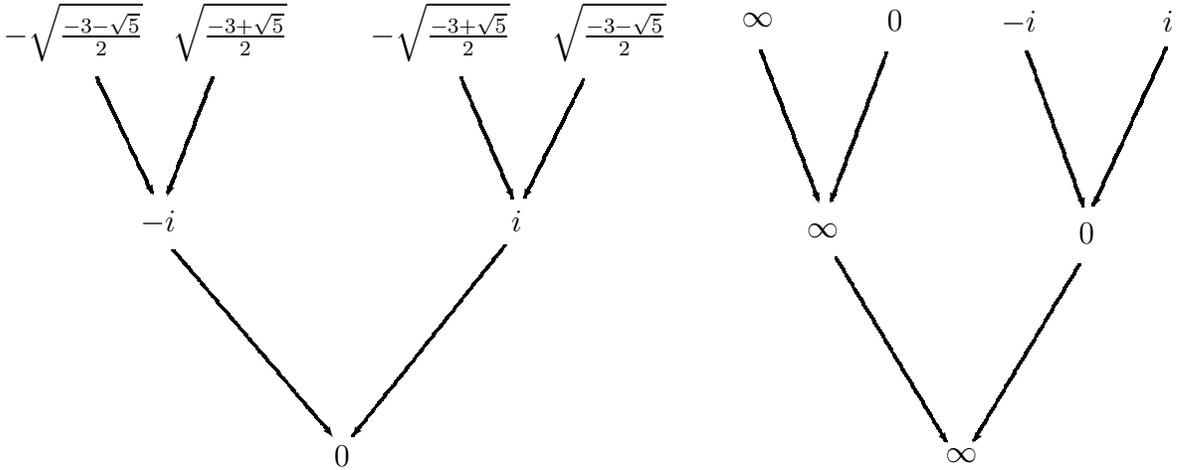

When $g(\alpha) = \infty$, using the notation of Figure \ref{fig: trees}, we have $V_1 = \{0\}$, and thus $V_n = \phi^{-(n-1)}(0)$.  It follows that the arboreal representation is the same as the case $g(\alpha) = 0$.  Thus to prove Conjecture~\ref{genmain}, we need only consider pairs of the form 
\[
(\phi, \alpha) = \left(k(z^2 + 1)/z, 0\right).
\]
Therefore to establish Conjecture~\ref{genmain}, it is enough to prove Conjecture~\ref{mainconj}. 
 

We return now to the general case and establish some notation.  Let $K$ be a number field, and $\phi : \mathbb{P}^1 \to \mathbb{P}^1$ a rational function defined over $K$.  Suppose that $\phi([X,Y]) = [P(X,Y), Q(X,Y)]$ in homogeneous coordinates, with $P(X,Y)$ and $Q(X,Y)$ having no common factors of positive degree.  Fix particular choices of $P$ and $Q$, and let 
\begin{align*}
\Phi(X,Y)\colon \mathbb{A}^2 &\to \mathbb{A}^2\\
(X,Y) & \mapsto (P(X,Y), Q(X,Y)) 
\end{align*}
be a lift of $\phi$.
  Define $P_n, Q_n \in K[X,Y]$ by 
\begin{align}
	\Phi^n(X,Y) &:= (P_n(X,Y), Q_n(X,Y)), \text{ where} \label{pndef} \\
P_n(X,Y) &= P_{n-1}(P(X,Y),Q(X,Y)) \text{ and } \nonumber\\
  Q_n(X,Y) &= Q_{n-1}(P(X,Y),Q(X,Y)), \text{ or equivalently}\nonumber \\
  P_n(X,Y) &= P(P_{n-1}(X,Y),Q_{n-1}(X,Y))  \text{ and }  \nonumber\\
   Q_n(X,Y) &= Q(P_{n-1}(X,Y),Q_{n-1}(X,Y)). \nonumber
\end{align}
Note that $\phi^n([X,Y]) = [P_n(X,Y), Q_n(X,Y)]$, though using homogeneous coordinates may involve cancellation of some common constant factors.

We use lower-case letters to denote de-homogenizations, and summarize our notation:

\vspace{0.15 in}

\begin{tabular}{ l l }
\vspace{0.07 in}
$\phi([X,Y]) = [P(X,Y), Q(X,Y)]$ & a rational map on $\mathbb{P}^1$ of degree $d$.\\
\vspace{0.07 in}
$\phi(x) = p(x)/q(x)$ & the dehomogenization of $\phi$.\\
\vspace{0.07 in}
$\Phi(X,Y) = (P(X,Y), Q(X,Y))$ & natural lift of $\phi$ to a map on $\mathbb{A}^2$. \\
\vspace{0.07 in}
$\Phi^n(X,Y) = (P_n(X,Y), Q_n(X,Y))$ & $n$th iterate of $\Phi$. \\
\vspace{0.07 in}
$p_n(x) = P_n(x,1)$, $q_n(x) = Q_n(x,1)$ & dehomogenized versions of $P_n, Q_n$. \\
\vspace{0.07 in}
$A_\phi = \left\{f \in \PGL_2(\overline K) : \phi\circ f = f\circ \phi \right\}$ & the automorphism group of $\phi$.\\
\vspace{0.07 in}
$A_{\phi, \alpha} = \left\{f \in A_\phi : f(\alpha) = \alpha\right\}$. & \\
\vspace{0.07 in}
$\ell(R)$ & leading coefficient of the polynomial $R$.\\
\vspace{0.07 in}
$d_R$ & the degree of $R$. \\
\vspace{0.07 in}
$V_n$ & $\phi^{-n}(0),$ unless $0$ is periodic; see~\eqref{eq : trees}.  \\
\vspace{0.07 in}
$T = \bigcup_{i \geq 1} V_n$ & preimage tree of $\phi$ with root $0$. \\
\vspace{0.07 in}
$T_n = \bigcup_{1 \leq i \leq n} V_n$ & truncation of $T$ to level $n$. \\
\vspace{0.07 in}
$K_n = K(V_n) = K(T_n)$. \\ 
\vspace{0.07 in}
$K_{\infty} = \bigcup_n K_n = K(T)$. \\
\vspace{0.07 in}
$G_n = \Gal(K_n/K)$. \\
\vspace{0.07 in}
$G_{\infty} = \Gal (K_{\infty}/K) = \invlim G_n$. \\
\end{tabular}

When we refer to a ``separable polynomial,'' we mean that the polynomial has distinct roots.
We also adopt the convention that 
$\alpha = 0$ and that $\infty$ does not appear in the pre-image tree of $\alpha = 0$; that is, we assume $\phi^n(\infty) \neq 0$.  These assumptions makes the statements and proofs of our results in Section~\ref{galgen} much simpler, and comes at no cost.  Indeed, as noted earlier in this section, the representations associated to $(\phi, \alpha)$ and $(\phi^g, g(\alpha))$ are the same for $g \in \PGL_2(K)$.  Choosing $g$ with $g(\alpha) = 0$  then reduces to the case $\alpha = 0$.  We may similarly require that $g(\beta) = \infty$, where $\beta \in K$ is any point disjoint from the preimage tree of $\alpha$.
  
Because $\alpha = 0$,  $G_n$ is the Galois group of the de-homogenized polynomial $p_n(x) = P_n(x,1) \in K[x]$.  We frequently move back and forth between $P_n(X,Y)$ and $p_n(x)$.  The de-homogenized version of the recursion for $P_n$ is
\begin{equation} \label{rec1}
p_n(x) = q(x)^{\deg P_{n-1}} p_{n-1}(p(x)/q(x)) = q(x)^{d^{n-1}} p_{n-1}(p(x)/q(x)), 
\end{equation}
or equivalently 
\begin{equation} \label{rec2}
p_n(x) = q_{n-1}(x)^{\deg P} p(p_{n-1}(x)/q_{n-1}(x)) = q_{n-1}(x)^{d} p(p_{n-1}(x)/q_{n-1}(x)),
\end{equation}
where in both cases $d = \deg P$.
These hold for all $x$ with $q(x) \neq 0$ and $q_{n-1}(x) \neq 0$, respectively.

\section{Discriminants, Irreducibility, and Galois Theory of rational functions} \label{galgen}

We begin with results concerning the discriminant of the numerator of an iterate of a rational function.  We then consider the case $d = 2$, prove results on the irreducibility of such polynomials, and then apply these results to the question of under what conditions the degree of the extension $[K_n : K_{n-1}]$ is as large as possible.  These results hold without consideration of the automorphism group $A_\phi$ of $\phi$.  In Section~\ref{galaut} we examine the case where $\#A_\phi = 2$.  

Throughout, we denote the degree of a polynomial $s(x) \in K[x]$ by $d_s$ and its leading coefficient by $\ell(s)$.  
We recall the {\em resultant} of two polynomials $g_1, g_2 \in K[x]$ may be defined as
$$\Res(g_1,g_2) = \ell(g_1)^{d_{g_2}}\prod_{g_1(\alpha) = 0} g_2(\alpha),$$
and is a homogeneous polynomial in the coefficients of $g_1$ and $g_2$ that vanishes if and only if $g_1$ and $g_2$ have a common root in $\overline{K}$.  We will also make use of the following basic equality:
\begin{equation} \label{classic}
\prod_{g_1(\alpha) = 0} g_2(\alpha) = \ell(g_2)^{d_{g_1}} \ell(g_1)^{-d_{g_2}} \prod_{g_2(\alpha) = 0} g_1(\alpha).
\end{equation}

 The discriminant of the polynomial $p_n$ will prove to be a  fundamental tool in what follows.  However,  $p_n$ is constructed as a de-homogenized polynomial $P_n$, and $P_n$ is given by a double-recursion. Standard results on calculating discriminants do not apply in this more complicated situation;  we need a new tool.  We begin with the somewhat simpler case of calculating the discriminant of a single polynomial that is the de-homogenization of a function of two other (homogeneous) polynomials.  In Theorem~\ref{discthm}, we  apply this result recursively to find a discriminant formula for $p_n$. 

\begin{lemma} \label{disclem}
Let $F, P, Q \in K[X,Y]$ be non-constant homogeneous polynomials with $\deg P = \deg Q = d$ and $P$ and $Q$ having no common roots in $\mathbb{P}^1(\overline{K})$.  Let $H(X,Y) = F(P(X,Y),Q(X,Y))$, and let $h,f,p,q\in K[x]$ denote the de-homogenizations of $H,F,P$, and $Q$, respectively.  Let $c = qp' - pq'$.  Finally, assume that $H(1,0) \neq 0$ and $F(1,0) \neq 0$.  
Then 
$$\Disc\ h = \pm \ell(h)^{k_1} \ell(q)^{k_2} \ell(f)^{k_3} \ell(c)^{k_4} (\Disc\ f)^d \left(\Res(q,p) \right)^{d_f(d_f - 2)} \prod_{c(\gamma) = 0} h(\gamma),$$
where 
\begin{align*}
k_1 &= d_fd - 2 - d_c - d_q(d_f - 2), &k_2 &= d_f(d - d_p)(d_f - 2),\\
k_3 &= (d_q - d)(d_f-2), \text{ and } &k_4 &= d_fd.
\end{align*}  
\end{lemma}

\begin{proof}
By definition, 
\begin{equation} \label{e1}
\Disc\ h = \pm \ell(h)^{-1}\Res(h, h') = \pm
\ell(h)^{{d_h} - 2} \prod_{h(\alpha) = 0} h'(\alpha).
\end{equation}
De-homogenizing $H$ gives $h(x) = q(x)^{d_F}f\left(p(x)/q(x)\right)$ provided $q(x) \neq 0$, and thus 
\begin{equation} \label{e2}
h' = d_F q^{d_F-1}q' \cdot f(p/q) + q^{d_F} f'(p/q) \cdot (p/q)',
\end{equation}
assuming $q(x) \neq 0$.  Since $F(1,0) \neq 0$, no root $\alpha$ of $h$ can satisfy $q(\alpha) = 0$.  This implies that for each $\alpha$ with $h(\alpha) = 0$, we have $f(p(\alpha)/q(\alpha)) = 0$, so the first summand in~\eqref{e2} vanishes when $x = \alpha$.   We may thus rewrite the right side of~\eqref{e1} as 
\begin{equation} \label{e3}
\pm \ell(h)^{{d_h} - 2} \left( \prod_{h(\alpha) = 0} q(\alpha) \right)^{d_F - 2}  \prod_{h(\alpha) = 0} (f' \circ p/q) (\alpha)  \prod_{h(\alpha) = 0} (p'q - q'p)(\alpha).
\end{equation} 

Using~\eqref{classic}, the first product in~\eqref{e3} is equal to $\left(\ell(q)^{d_h}  \ell(h)^{-d_q} \prod_{q(\pi) = 0} h(\pi)\right)^{d_F-2}$.  Moreover, for each root $\pi$ of $q$ we have
\[
	h(\pi) = H(\pi,1) = F \left( P(\pi,1), Q(\pi,1) \right) = F \left( P(\pi,1), 0 \right) = \ell(f) p(\pi)^{d_F}.
\]
Hence the first product in~\eqref{e3} becomes 
\begin{align*}
	\Bigg(\ell(q)^{d_h} \ell(h)^{-d_q} \ell(f)^{d_q} \prod_{q(\pi) = 0}  & p(\pi)^{d_f}\Bigg)^{{d_F}-2}\\
	&=
	\left(\ell(q)^{d_h} \ell(h)^{-d_q} \ell(f)^{d_q} \left(\ell(q)^{-d_p} \Res(q,p)\right)^{d_f}\right)^{{d_F}-2}.
\end{align*}

Turning to the second product in~\eqref{e3}, we have already noted that $h(\alpha) = 0$ implies $p(\alpha)/q(\alpha)$ is a root of $f$.  Moreover, for each root $\beta$ of $f$, there are with multiplicity precisely $d$ elements $\alpha$ with $p(\alpha)/q(\alpha) = \beta$ (this is ensured by the assumption that $H(1,0) \neq 0$).  Thus as $\alpha$ runs over all roots of $h$, $(p/q)(\alpha)$ runs over all roots of $f$, hitting each one $d$ times.  Hence the second product in~\eqref{e3} equals 
\[
	\left(\prod_{f(\beta) = 0} f'(\beta)\right)^d = \left(\ell(f)^{-(d_{f}-2)} \Disc (f)\right)^d.
\]

From~\eqref{classic}, the third product in~\eqref{e3} equals 
\[
	\ell(c)^{d_h} \ell(h)^{-d_c} \prod_{c(\gamma) = 0} h(\gamma).
\]

Gathering the terms containing $\ell(h)$ and $\ell(f)$, and using the fact that $d_f = d_F$ (since $F(1,0) \neq 0$) and $d_h = dd_f$ (since $H(1,0) \neq 0$) completes the proof.  
\end{proof}

\begin{remark}
Note the conditions  $H(1,0) = 0$ and $F(1,0) = 0$ correspond to our  assumption that $\infty$ is not in the pre-image tree of $\alpha = 0$.  These assumptions greatly ease an (already complicated) calculation, but 
it is certainly possible to obtain similar formulas in the case that $H(1,0) = 0$ or $F(1,0) = 0$.  In these cases $F$ factors as a product $F_1F_2$ with $F_2(1,0) \neq 0$ and $H_2(1,0) \neq 0$, where $H_2 := F_2(P,Q)$.  One then splits the product on the right side of~\eqref{e1} into the product over the roots of $h_2$ and the product over the remaining roots of $h$.  
\end{remark}

\begin{theorem} \label{discthm}
Let $\phi = p(x)/q(x) \in K(x)$ be a rational function of degree $d \geq 2$, let $n \geq 2$, and define $p_n$ and $q_n$ recursively so that $\phi^n = p_n(x)/q_n(x)$.  Let $c = qp' - pq'$.  Assume that $\phi^n(\infty) \neq 0$ and $\phi^{n-1}(\infty) \neq 0$.  If $\phi(\infty) \neq \infty$, then
\begin{equation} \label{discform1}
\Disc\ p_n = \pm \ell(p_n)^{k_1} \ell(q)^{k_2} \ell(c)^{k_3} (\Disc\ p_{n-1})^d \left(\Res(q,p) \right)^{d^{n-1}(d^{n-1} - 2)} \prod_{c(\gamma) = 0} p_n(\gamma),
\end{equation}
where 
\[
k_1 = 2d - 2 - d_c , \qquad k_2 = d^{n-1}(d - d_p)(d^{n-1} - 2), \text{ and } \qquad k_3 = d^n.
\]
If $\phi(\infty) = \infty$, then
\begin{equation} \label{discform2}
\Disc\ p_n = \pm \ell(p)^{k_1} \ell(c)^{k_2} (\Disc\ p_{n-1})^d \left(\Res(q,p) \right)^{d^{n-1}(d^{n-1} - 2)} \prod_{c(\gamma) = 0} p_n(\gamma),
\end{equation}
where 
\begin{equation}
	k_1 = d^{2n-1} - d_q(d^{2n-2} - 2d^{n-1}) - d_c(1 - d^n)/(1-d) - 2
	 \text{ and }
	 \qquad k_2 = d^n.\label{k1value}
\end{equation}
\end{theorem}

\begin{remark}
Suppose that $\phi$ is post-critically finite, so the forward orbit of each $\gamma$ with $c(\gamma) = 0$ is finite.  An induction on equation~\eqref{rec1} shows that the set of primes dividing  $p_n(\gamma)$ for all critical points $\gamma$ and all $n\geq 1$ is finite.  If  $\phi(\infty) = \infty$, induction on equation~\eqref{discform2} then shows that the set of primes dividing $\Disc (p_n)$ for any $n$ is likewise finite.

When $\phi(\infty) \neq \infty$ and $d_c = 2d - 2$ (in other words, when $\infty$ is not a critical point), the term $\ell(p_n)$ does not contribute to the product in equation~\eqref{discform1}, so the set of primes dividing $\Disc(p_n)$ for any $n$ is finite in this case as well.
When $\phi(\infty) \neq \infty$ and $\infty$ is a critical point, then from~\eqref{discform1}, we have that $\Disc\ p_n$ is divisible by $\ell(p_n)$.  However, 
$\ell(p_n) = P_n(1,0)$, and since by assumption $\phi$ is post-critically finite and $\infty$ is a critical point, 
$P_n(1,0)$ can take on only finitely many values as $n$ varies.  

Hence,  Theorem~\ref{discthm} shows that if $\phi$ is post-critically finite, then there is a finite set of primes $S$ such that for every $n\geq 1$, $\Disc\ p_n$ is divisible only by primes in $S$.   In this case, then, the field $K_{\infty}$ is ramified over only finitely many primes of $K$.  (For a generalization of this result, see \cite{HC}.)  
\end{remark}

\begin{proof}
Since $\phi^n(\infty) \neq 0$, and $\phi^{n-1}(\infty)  \neq 0$, we have $P_n(1,0) \neq 0$ and $P_{n-1}(1,0) \neq 0$, respectively.  We may thus apply Lemma~\ref{disclem} with $F = P_{n-1}$ and $H = P_n$.  Note also that $\phi^{n-1}(\infty) \neq 0$ implies that $\deg p_{n-1} = d^{n-1}$.  Moreover, if we assume that $\phi(\infty) \neq \infty$, then $d_q = d$.  Lemma~\ref{disclem} then immediately gives formula~\eqref{discform1}.

Assuming now that $\phi(\infty) = \infty$, we have $d_p = d$, which kills the $\ell(q)$ term in Lemma~\ref{disclem}.  Moreover, $\phi(\infty) = \infty$ also implies that for all $k$, $\ell(p_k) = \ell(p)\ell(p_{k-1})^d$, and an induction gives $\ell(p_k) = \ell(p)^{(1-d^k)/(1-d)}$.  
The power of $\ell(p)$ in the expression in Lemma~\eqref{disclem} is thus
$$\frac{(1-d^n)(d^n - 2 - d_c - d_q(d^{n-1} - 2))}{1-d} + \frac{(1-d^{n-1})(d_q - d)(d^{n-1}-2)}{1-d}.$$
This simplifies to the value of $k_1$ given in equation~\eqref{k1value}.  
\end{proof}

We now consider the irreducibility of the $p_n$ in the case of quadratic rational functions.  

\begin{lemma} \label{discprop}
Let $\phi(x) \in K(x)$ have degree~$2$, let $t$ be a parameter, and let $\gamma_1, \gamma_2 \in \mathbb{P}^1(\overline{K})$ be the critical points of $\phi$.  Then there exists $C \in K$ such that 
\begin{equation} \label{discpropeqn}
\Disc_x \left(p(x) - tq(x) \right) = 
C \prod_{\phi(\gamma_i) \neq \infty} \left( \phi(\gamma_i) - t \right).
\end{equation}
Moreover, if $p(x)$ is separable, then
 $C = \Disc\ p(x) \cdot \displaystyle\prod_{\phi(\gamma_i) 
 \neq \infty} \left( \phi(\gamma_i)^{-1} \right)$.
\end{lemma}

\begin{remark} This is a special case of a more general phenomenon; see \cite[Proposition 1]{HC}.
\end{remark}

\begin{proof}
Note that both sides of~\eqref{discpropeqn} are in $K[t]$ (the right side because the $\phi(\gamma_i)$ are either rational or Galois-conjugate).  We show that the roots of $\Disc_x \left(p(x) - tq(x)\right)$ in $\overline{K}$ are precisely the $\phi(\gamma_i)$ with $\phi(\gamma_i) \neq \infty$, and this is enough to establish the Lemma.  

Note first that since $\phi$ is quadratic, by the Riemann-Hurwitz formula we must have $\gamma_1 \neq \gamma_2$ and $\phi(\gamma_1) \neq \phi(\gamma_2)$.   For a given $t \in \overline{K}$, the degree-two homogeneous polynomial 
$P(X,Y) - tQ(X,Y)$ has a single root in $\mathbb{P}^1(\overline{K})$ if and only if $\phi^{-1}(t)$ has a single element, which occurs precisely when $t = \phi(\gamma_i)$.  Thus $\Disc \left( P(X,Y) - tQ(X,Y)\right)$ vanishes if and only if $t = \phi(\gamma_i)$, and since $t \in \overline{K}$ we cannot have $t = \infty$.  
Finally, both $P(X,Y) - tQ(X,Y)$ and $p(x) - tq(x)$ have degree~$2$, and thus $\Disc \left(P(X,Y) - tQ(X,Y) \right)$ is the same as $\Disc_x \left(p(x) - tq(x)\right) $. 

The last statement of the proposition comes from setting $t = 0$ in~\eqref{discpropeqn} and noting that the separability of $p(x)$ implies $\Disc\ p(x)$ does not vanish, and therefore $\prod_{\phi(\gamma_i) \neq \infty} \left( \phi(\gamma_i) \right)$ cannot vanish either.
 \end{proof}

\begin{lemma} \label{irredlem}
Let $\phi(x) \in K(x)$ have degree~$2$, and for each $i \geq 0$, denote by $K_i$ the splitting field of $p_i$.  Assume that $p_{n-1}$ is irreducible in $K[x]$, let $\alpha$ be a root of $p_{n-1}$, and let $\gamma_1, \gamma_2 \in \mathbb{P}^1(\overline{K})$ be the critical points of $\phi$. Then there exists $C \in K$ such that $p_n$ is irreducible in $K[x]$ if and only if 
$$C \prod_{\phi(\gamma_i) \neq \infty} \left( \phi(\gamma_i) - \alpha \right) \not\in K_{n-1}^{*2}.$$
If $p(x)$ is separable, then $C$ has the same value as in Lemma~$\ref{discprop}$.
\end{lemma}


\begin{proof}
Denote the roots of $p_{n-1}$ by $\alpha_1, \ldots, \alpha_r$, and take $\alpha_1 = \alpha$.  Since $p_{n-1}$ is irreducible, the $\alpha_i$ are Galois conjugates, and hence the action of Galois on the roots of $p_n$ is either as a single orbit of $2r$ elements (and thus $p_n$ is irreducible) or two orbits of $r$ elements.  The latter case holds if and only if each orbit contains exactly one element in each fiber $\phi^{-1}(\alpha_i)$, or equivalently if and only if the roots of $p(x) - \alpha_i q(x)$ are not conjugate for any $i$.  This holds if and only if the $p(x) - \alpha_i q(x)$ are all reducible over $K_{n-1}$.  Because the $\alpha_i$ are all conjugate, this is equivalent to the reducibility of $p(x) - \alpha q(x)$ over $K_{n-1}$,  which occurs precisely when $\Disc \left(p(x) - \alpha q(x)\right)$ is a square in $K_{n-1}$.  The result now follows from Lemma~\ref{discprop}.  
\end{proof}

Recall that $G_n$ is the Galois group of the splitting field of the polynomials $p_n$.  To understand when $G_n$ is as large as possible, it will be necessary to have conditions under which the polynomials $p_n$ are irreducible.  We now give a criterion for the irreducibility of $p_n$ assuming that $p_{n-1}$ is irreducible and has even degree.   Note that the criterion here is sufficient but not necessary.  The result is useful in that it applies to all degree~$2$ rational maps, but  unfortunately the hypotheses are not  satisfied in the case of quadratic maps with a nontrivial automorphism.  We will need a refinement of this result in that case, which we provide in Theorem~\ref{irredthm}.

\begin{theorem} \label{irredcritgen}
Let $\phi(x) \in K(x)$ have degree~$2$.  Suppose that $n \geq 2$, and that $p_{n-1}$ is irreducible in $K[x]$ and has even degree.  Let $\ell(p_{n-1})$ be the leading coefficient of $p_{n-1}$, let $\gamma_1, \gamma_2 \in \mathbb{P}^1(\overline{K})$ be the critical points of $\phi$, and without loss say $\phi(\gamma_1) \neq \infty$.  If $\phi(\gamma_2)$ is not \textup(resp. is\textup) $\infty$, then $p_{n}$ is irreducible in $K[x]$ provided
\begin{equation} \label{tmbg}
p_{n-1}(\phi(\gamma_1)) \cdot p_{n-1}(\phi(\gamma_2))  \not\in K^{*2} \qquad ( \textrm{resp. } \ell(p_{n-1}) \cdot p_{n-1}(\phi(\gamma_1))  \not\in K^{*2}).
\end{equation}
\end{theorem}

\begin{remark} The condition that $p_{n-1}$ have even degree is implied by $\phi^{n-1}(\infty) \neq 0$.  Moreover, if $\gamma_i \neq \infty$, then from~\eqref{rec1} and the assumption $n \geq 2$,  $p_{n-1}(\phi(\gamma_i)) = p_n(\gamma_i)$ up to squares.  Thus if both $\gamma_1$ and $\gamma_2$ are finite, then \eqref{tmbg} becomes $$p_{n}(\gamma_1)p_{n}(\gamma_2)  \not\in K^{*2} \qquad  ( \textrm{resp. } \ell(p_{n-1}) p_{n}(\gamma_1)  \not\in K^{*2}).$$
\end{remark}

\begin{proof}
By Lemma~\ref{irredlem}, we must show 
\[
	C \prod_{\phi(\gamma_i) \neq \infty} \left( \phi(\gamma_i) - \alpha \right) \not\in K_{n-1}^{*2},
\]
for some root $\alpha$ of $p_{n-1}$, and for this it is sufficient to show that the norm of the left side as an element of $K_{n-1}/K$ is not a square in $K$.   This norm equals
\[
	C^{\deg p_{n-1}} \prod_{i=1}^2 \prod _{j = 1}^r \left( \phi(\gamma_i) - \alpha_j \right) \qquad \left( \text{resp. } C^{\deg p_{n-1}} \prod _{j = 1}^r \left( \phi(\gamma_1) - \alpha_j \right) \right),
\]
where $\alpha_1, \ldots, \alpha_r$ denote the roots of $p_{n-1}$.  This is the same as
\[
	C^{\deg p_{n-1}} \ell(p_{n-1})^{-2} \prod_{i = 1}^2 p_{n-1}(\phi(\gamma_i)) \qquad \left( \text{resp. } C^{\deg p_{n-1}} \ell(p_{n-1})^{-1} p_{n-1}(\phi(\gamma_i)) \right).
\]
Since $\deg p_{n-1}$ is even, $C^{\deg p_{n-1}}$ is a square.   
\end{proof}

For the remainder of this section, we let $n$ be fixed, and assume that $\phi^n(\infty) \neq 0$ and $\phi^{n-1}(\infty) \neq 0$.  We also assume that $p_n$ is separable, which by Theorem~\ref{discthm} is equivalent to $\phi^i(\gamma) \neq 0$, $i = 1, \ldots, n$ for all critical points $\gamma$ of $\phi$.  Together, these assumptions imply that there are $d^n$ distinct roots $\alpha_1, \ldots, \alpha_{d^n}$ of $p_n$, and 
$d^{n-1}$ distinct roots $\beta_1, \ldots, \beta_{d^{n-1}}$ of $p_{n-1}$.  Moreover, the $\alpha_i$ are precisely the roots of $p(x) - \beta_j q(x), $ for $ j = 1, \ldots, d^{n-1}$.  Recall that 
$K_n = K(\alpha_1, \ldots, \alpha_{d^n})$ and $K_{n-1} = K(\beta_1, \ldots, \beta_{d^{n-1}})$.  

We examine the extension $K_n/K_{n-1}$ in the case $d = 2$ and give conditions that ensure it is as large as possible.
Recall that $G_n = \Gal(K_n/K)$.  The assumptions of the previous paragraph imply an injection $G_n \hookrightarrow \Aut(T_n)$, where $T_n$ is the complete binary rooted tree of height $n$.  Restriction gives a homomorphism 
$\Aut(T_n) \to \Aut(T_{n-1})$, whose kernel is generated by the transpositions swapping a single pair of vertices at level~$n$, both connected to a given vertex at level $n-1$.  Thus the kernel is isomorphic to $(\Z/2\Z)^{2^{n-1}}$.  Hence $\Gal(K_n/K_{n-1})$ must inject into this group.  We now show how one can see this directly from the way that $K_n$ is constructed from $K_{n-1}$;  this point of view will also be the most useful for establishing our maximality results.  

Because $\deg \phi = 2$, $p(x) - \beta_j q(x) \in K_{n-1}(x)$ is a quadratic polynomial.  Note that $K_n$ is obtained from $K_{n-1}$ by adjoining the roots of $p(x) - \beta_j q(x)$ for $j = 1, \ldots, 2^{n-1}$, so we have that $K_n$ is a $2$-Kummer extension of $K_{n-1}$, and indeed letting 
\begin{equation}
\delta_j = \Disc (p(x) - \beta_j q(x)),\label{deltajdef} \text{ we have }
K_n = K_{n-1}\left(\sqrt{\delta_j} : j = 1, \ldots, 2^{n-1}\right).
\end{equation}
It follows that $\Gal(K_n/K_{n-1}) \hookrightarrow (\Z/2\Z)^{2^{n-1}}.$  Using Kummer theory (e.g. \cite[Section VI.8]{langalg}), $[K_n : K_{n-1}]$ is the order of the group $D$ generated by the classes of the $\delta_j$ in 
$K_{n-1}^*/K_{n-1}^{*2}$.  Now,
\[
\#D = \frac{2^{2^{n-1}} }{ \#V}, 
\text{ where }
V = \{(e_1, \ldots, e_{2^{n-1}}) \in 
\mathbb{F}_2^{2^{n-1}} : \prod_j \delta_j^{e_j} \in K_{n-1}^{*2} \}.
\]
That is, $V$ is the group of relations among the $\delta_j$.  One sees easily that $V$ is an $\mathbb{F}_2$-vector space, and that the action of $G_{n-1} := \Gal(K_{n-1}/K)$ on the $\delta_j$ gives an action of $G_{n-1}$ on $V$ as linear transformations.  It follows that $V$ is an $\mathbb{F}_2[G_{n-1}]$-module.

The following Lemma is due to M. Stoll \cite{stoll}.  We give the proof here for the sake of completeness.

\begin{lemma}[Stoll] \label{stolllem}
Let $\Gamma$ be a $2$-group and $M \neq 0$ a $\mathbb{F}_2[\Gamma]$-module.  Then the submodule $M^{\Gamma}$ of $\Gamma$-invariant elements is non-trivial.
\end{lemma}

\begin{proof}
Induct on $\# \Gamma$.  Suppose $\Gamma = \{e, \sigma\} \cong \Z/2\Z$, and take $m \in M$ with $m \neq 0$.  Then either $\sigma(m) = m$ or $m + \sigma(m) \neq 0$ (since $M$ is an $\mathbb{F}_2$-module).  In the former case, $m$ is a nontrivial element of $M^{\Gamma}$, while in the latter case 
$m + \sigma(m)$ is a nontrivial element of $M^{\Gamma}$.

If $\# \Gamma > 2$, then let $N$ be a nontrivial normal subgroup of $\Gamma$ (possible since $\Gamma$ is a 2-group).  Then $M$ is an $\mathbb{F}_2[N]$-module also, so by induction $M^N \neq 0$.  However, $M^N$ is an $\mathbb{F}_2[\Gamma/N]$-module, so again by induction 
$0 \neq (M^N)^{\Gamma/N} = M^{\Gamma}$.  
\end{proof}

We now give a condition that will guarantee the extension $[K_n : K_{n-1}] $ is as large as possible.  \emph{A priori}, this result depends on deciding whether an element  of  $K_{n-1}$ is a square.  However,  we can actually give a condition ensuring $[K_n : K_{n-1}] = 2^{2^{n-1}}$ solely in terms of the arithmetic of $K$.   We provide such a condition in Corollary~\ref{maxcor2}. 

\begin{theorem} \label{maxthm1}
Let $\phi \in K(x)$ have degree~$2$ with $\phi^n(\infty) \neq 0$ and $\phi^{n-1}(\infty) \neq 0$.  Suppose that $n \geq 2$ and that $p_{n-1}$ is irreducible in $K[x]$.  Let  $\ell(p_{n-1})$ be the leading coefficient of $p_{n-1}$, let  $\gamma_1, \gamma_2 \in \mathbb{P}^1(\overline{K})$ be the critical points of $\phi$, and without loss say $\phi(\gamma_1) \neq \infty$.  If $\phi(\gamma_2)$ is not \textup(resp. is\textup) $\infty$, then 
$[K_n : K_{n-1}] = 2^{2^{n-1}}$ if and only if 
\begin{equation} \label{seccond}
p_{n-1}(\phi(\gamma_1))  p_{n-1}(\phi(\gamma_2))  \not\in K_{n-1}^{*2} \qquad ( \textrm{resp. } \ell(p_{n-1})  p_{n-1}(\phi(\gamma_1))  \not\in K_{n-1}^{*2}).
\end{equation}
\end{theorem}

\begin{remark}  As in Theorem \ref{irredcritgen}, if both $\gamma_1$ and $\gamma_2$ are finite, then it follows from~\eqref{rec1} and $n \geq 2$ that \eqref{seccond} may be replaced by \begin{equation}\label{eqn:element of K}
p_{n}(\gamma_1)p_{n}(\gamma_2)  \not\in K_{n-1}^{*2} \qquad  ( \textrm{resp. } \ell(p_{n-1}) p_{n}(\gamma_1)  \not\in K_{n-1}^{*2}).
\end{equation}
\end{remark}

\begin{proof}
From the discussion immediately preceding Lemma~\ref{stolllem}, we have 

\[
[K_n : K_{n-1}] = 2^{2^{n-1}}/ \#V, \text{ where }
V = \{(e_1, \ldots, e_{2^{n-1}}) \in \mathbb{F}_2^{2^{n-1}} : \prod_j \delta_j^{e_j} \in K_{n-1}^{*2} \}
\] has a natural structure of a $\mathbb{F}_2[G_{n-1}]$-module.  Thus $[K_n : K_{n-1}] < 2^{2^{n-1}}$ if and only if $V \neq 0$, which by Lemma~\ref{stolllem} occurs if and only if $V^{G_{n-1}} \neq 0$.  However, since $p_{n-1}$ is irreducible, $G_{n-1}$ acts transitively on the $\delta_j$ defined in~\eqref{deltajdef}, implying that the only possible nontrivial element in $V^{G_{n-1}}$ is $(1, 1, \ldots, 1)$.  Hence
$[K_n : K_{n-1}] < 2^{2^{n-1}}$ if and only if 
$$\prod_{j=1}^{2^{n-1}} \Disc (p(x) - \beta_j q(x)) \in K_{n-1}^{*2},$$
where as before the $\beta_j$ are the $2^{n-1}$ distinct roots of $p_{n-1}$.
By Lemma~\ref{discprop}, this is equivalent to 
\[
	 \prod_{i=1}^2 \prod _{j = 1}^{2^{n-1}} \left( \phi(\gamma_i) - \beta_j \right) \in K_{n-1}^{*2} \qquad \left( \text{resp. }  \prod _{j = 1}^{2^{n-1}} \left( \phi(\gamma_1) - \beta_j \right) \in K_{n-1}^{*2} \right).
\]
The theorem now follows from the fact that $p_{n-1}(x) = \ell(p_{n-1}) \prod (x - \beta_j)$.
\end{proof}

Let $s=2$ if  $\phi(\gamma_2) \neq \infty$ and $s=1$ otherwise.  Then equation~\eqref{eqn:element of K} can be summarized as
\[
 \ell(p_{n-1})^{s} \prod_{i=1}^s p_{n}(\gamma_i)  \not\in K_{n-1}^{*2}.
 \]
Since $\ell(p_{n-1})^{s} \prod_{i=1}^s p_{n}(\gamma_i)$  is not just an element of $K_{n-1}$ but also an element of $K$,  it cannot be a square in $K_{n-1}$ unless all the primes of $K_{n-1}$ that divide it lie over primes of $K$ that ramify in $K_{n-1}$.  Thanks to Theorem~\ref{discthm}, we have a handle on the primes that can ramify in $K_{n-1}$, and thus we can use Theorem~\ref{maxthm1} to give a simpler condition ensuring that $[K_n : K_{n-1}] = 2^{2^{n-1}}$.
In the Corollary, we limit ourselves for simplicity of statement to the case when $\infty$ is distinct from the critical points and values of $\phi$, and does not have $0$ in its forward orbit.  Denote by $v_{\p}$ the $\p$-adic valuation at a prime $\p$ in the ring of integers of $K$.  

\begin{corollary} \label{maxcor2}
Let $\phi = p(x)/q(x) \in K(x)$ have degree~$2$, let $c = qp'-pq'$, let $\ell(p_{n-1})$ be the leading coefficient of $p_{n-1}$, and suppose that $\phi^n(\infty) \neq 0$ for all $n  \geq 1$ and that $\phi$ has two finite critical points $\gamma_1, \gamma_2$ with $\phi(\gamma_i) \neq \infty$ for each $i$.  Suppose further that there exists a prime $\p$ of $K$ with $v_{\p}(p_{n}(\gamma_1)p_{n}(\gamma_2))$ odd 
and 
\begin{equation} \label{noram}
0 = v_{\p} (\ell(p)) = v_\p (\ell(c)) = v_\p (\Res (q,p)) =  v_\p(\Disc\ p) = v_\p (p_j(\gamma_i)) 
\end{equation}
for $1 \leq i \leq 2, 2 \leq j \leq n-1$.
Then $[K_n : K_{n-1}] = 2^{2^{n-1}}$.
\end{corollary}

\begin{proof}
If $\phi(\infty) = \infty$, the conditions in~\eqref{noram}, along with~\eqref{discform2} and induction, imply $v_\p(\Disc (p_{n-1})) = 0$.  If $\phi(\infty) \neq \infty$, then the conditions in \eqref{noram} and \eqref{discform1} give the same conclusion (note that $\infty$ not a critical point and $\phi(\infty) \neq 0$ imply that $k_1 = k_2 = 0$ in \eqref{discform1}).  Hence $\p$ does not ramify in $K_{n-1}$.  Therefore there is a prime $\P$ in the ring of integers of $K_{n-1}$ with $v_{\P}(\p)$ odd, and it follows that 
$v_\P(p_{n}(\gamma_1)p_{n}(\gamma_2))$ 
is odd, so $p_{n}(\gamma_1)p_{n}(\gamma_2)$ 
cannot be a square in $K_{n-1}^*$.  The corollary now follows from Theorem~\ref{maxthm1} and the remark preceding it.
\end{proof}

Corollary \ref{maxcor2} provides a convenient method for checking the maximality of $G_n$ for an arbitrary quadratic $\phi$, at least for small $n$. In certain circumstances, it can even be used to determine $G_\infty$, although the difficulties in disentangling possible interactions of the two critical orbits at various primes are considerable. We illustrate with the family of quadratic rational functions
$$
\phi_a(x) = \frac{1 + ax + (3+a) x^2}{1 - (4+a)x - (a+1)x^2}, \qquad a \in \Q \smallsetminus \{-2\}.
$$
The critical points of $\phi_a$ are $1$ and $-1/3$; in addition,  $\phi_a$ has the two-cycle $1 \mapsto -1 \mapsto 1$, and $\phi_a$ sends $0$ to $1$. The behavior of one critical orbit is thus quite simple, and the fact that $0$ is preperiodic ensures that elements of the other critical orbit are close to relatively prime. Define the polynomials $p$ and $q$ by $\phi = p/q$; then 
\[
 \Res (q,p) = 16(a+2)^2.
\]
 For $k \geq 2$, we have the recursion
\begin{equation} \label{recc}
p_k = q_{k-1}^2 + aq_{k-1}p_{k-1} + (3 + a)p_{k-1}^2, \qquad q_k = q_{k-1}^2 - (4+a)q_{k-1}p_{k-1} - (a+1)p_{k-1}^2.
\end{equation}
Note that the only primes $\ell$ where we might have 
\[
p_k \equiv 0 \pmod \ell
\quad \text{ and } \quad
q_k \equiv 0 \pmod \ell
\]
 are those dividing $2(a+2)$. The reason is that if $\ell \nmid \Res(q,p)$, then $\phi$ has \textit{good reduction} at $\ell$, and thus so do all iterates of $\phi$ \cite[Theorem 2.18]{jhsdynam}, implying that $\ell \nmid \Res(p_k, q_k)$. We remark that one can also apply Lemma~\ref{disclem} to the polynomials $p_k$, $q_k$, and $p_kq_k$ to obtain an exact formula for $\Res(p_k, q_k)$, which turns out to be a power of $\Res(p,q)$. 

Now let $t \in \Q$, and suppose that for some $k \geq 1$, we have $p_k(t) = q_k(t)$. An  induction shows that 
 \begin{align*}
 p_{k+i}(t) &= (4 + 2a)(p_{k+i-1}(t))^2
&& \text{ if }  i  \text{ is odd,
  and }\\
  p_{k+i}(t) &= 4(p_{k+i-1}(t))^2   &&\text{ if }  i  \text{ is even.}
  \end{align*}
It follows that there are positive integers $r_i, s_i$ with 
\begin{equation} \label{oneco}
p_{k+i}(t) = p_k(t)^{2^i}2^{r_i}(2+a)^{s_i},
\end{equation}
where $r_i \equiv s_i \bmod{2}$.

 Suppose  that $\ell \nmid 2(a+2)$ satisfies $\ell \mid p_k(t)$ for some $k \geq 1$, and take $k$ minimal with this property.  Then since $\ell \nmid q_k(t)$,
 \[
 p_{k+1}(t) \equiv q_k(t)^2 \not\equiv 0 \pmod{\ell}
\quad \text{ and } \quad
 q_{k+1}(t) \equiv p_{k+1}(t) \pmod{\ell}.
 \]
  It then follows from applying \eqref{oneco} with $k+1$ in place of $k$ that $\ell \nmid p_{k+i}(t)$ for all $i \geq 1$.

\begin{proof}[Proof of Thoerem~$\ref{nonpoly}$]
Consider the specialization $a=0$, so 
\[
\phi(x) = \frac{1 + 3 x^2}{1 - 4x - x^2}.
\]
 By the above analysis, any odd prime divides at most one term of the sequence
$p_n(-1/3)$. Moreover, $p_1(1) = 4$ and $p_2(1) = q_2(1) = 2^6$, and thus it follows from \eqref{oneco} that $p_n(1)$ is an even power of $2$ for all $n \geq 1$.  

We also note that $(p_n(-1/3), q_n(-1/3)) \in (\Z/5\Z)^2$ for $n \geq 1$ gives the orbit
$$(3,0) \mapsto (2,1) \mapsto (3,4) \mapsto (3,4) \mapsto \cdots$$
and thus neither of $\pm p_n(-1/3)$ is a square for all $n \geq 1$. Hence for each $n$ there is a prime at which $p_n(-1/3)$ --- and therefore $p_n(-1/3)p_n(1)$ --- has odd valuation. To apply Corollary \ref{maxcor2}, we need to show that for each $n$, this prime is not $2$ or $3$, since the leading coefficient of $c$ and $\Disc\ p$ are both $-12$.  Note, however, that  that we don't need to consider $n =1$, since clearly $[K_1:K] = 2$.

 Consider first the $3$-adic behavior of $p_n(-1/3)$. We have $p(-1/3) = 4/3$ and 
$q(-1/3) = 20/9$. From \eqref{recc}, we see that for $k \geq 2$, 
\[
-1 \geq v_3(p_{k-1}(-1/3)) > v_3(q_{k-1}(-1/3)),
\] 
which implies that 
\[
v_3(p_{k}(-1/3)) = 2v_3(q_{k-1}(-1/3)) = v_3(q_{k}(-1/3)).
\]
 Hence $v_3(p_{n}(-1/3))$ is even for all $n \geq 2$.  
 
 Turning now to the $2$-adic perspective, suppose that for some $k \geq 2$, 
 \[
 1 \leq e = v_2(p_{k-1}(-1/3)) = v_2(q_{k-1}(-1/3)),\]
  and write 
  \[
  p_{k-1}(-1/3) = 2^eu \text{ and }q_{k-1}(-1/3) = 2^ew, \text{ where } v_2(u) = v_2(w) = 0.
  \]
   We then have
$$p_{k} \equiv 2^{2e}(u^2 + 3v^2) \bmod{2^{2e + 3}} \equiv 2^{2e+2} \bmod{2^{2e + 3}},$$
and similarly for $q_{k}$. It follows that $v_2(p_{k}(-1/3)) = v_2(q_{k}(-1/3)) = 2e+2$, and since 
$v_2(p(-1/3)) = v_2(q(-1/3)) = 2$, we thus have that  
$v_2(p_n(-1/3))$ is even for all $n \geq 1$. 

Finally, we must show that $\phi_0^n(\infty) \neq 0$ for all $n \geq 1$. But $\phi_0(\infty) \equiv 0 \bmod{3}$, implying that $\infty$ maps modulo 3 into the 2-cycle $1 \mapsto -1 \mapsto 1$.  We have thus shown that when $a = 0$, $G_\infty \cong \Aut(T)$. 
\end{proof}




\section{Discriminants, Irreducibility, and Galois Theory of quadratic rational functions with an order-2 automorphism} \label{galaut}

In this section, we consider the setting of Conjecture~\ref{mainconj}, namely $\phi(x) = k(x^2 + b)/x$.  In the interest of describing exactly the arboreal Galois representation associated to such a map, we choose not to take $b = 1$, since doing so implies that conjugation by $x \to x/\sqrt{b}$ is defined over $K$, introducing a possible additional quadratic extension. Note that $\phi^n(\infty) \neq 0$ and  $\phi^{n-1}(\infty) \neq 0$.  Let $\iota(x) = -x$, and note that $\iota$ acts on the roots of $p_n$ without fixed points, since $0$ and $\infty$ are the only fixed points of $\iota$ and neither maps to $0$ under any iterate of $\phi$.

We wish to apply the same general program from Section~\ref{galgen} to this case.  However, Theorem~\ref{irredcritgen} and Theorem~\ref{maxthm1} do not apply, since the critical points satisfy $\gamma_1 = -\gamma_2$ and~$p_n$ is always an even function.  Hence
$p_{n}(\gamma_1)p_{n}(\gamma_2) = p_n(\gamma_1)^2$ is a square in $K_{n-1}^{*}$ for all~$n$.  Indeed, we will show that $[K_n : K_{n-1}] \neq 2^{2^{n-1}}$ for all $n \geq 2$.  
As in Section~\ref{galgen}, we have $G_n \hookrightarrow \Aut(T_n)$, and $T_n$ is the complete binary rooted tree of height $n$, provided that $p_n$ is separable.  However, now the image of $G_n$ must commute with the action of $\iota$ on $T_n$.  We thus have $G_n \subseteq C_n,$ where $C_n$ denotes the centralizer in $\Aut(T_n)$ of the element corresponding to the action of $\iota$.  As in Section \ref{intro}, $C_{\infty}:= \invlim C_n$ plays roughly the role of a Cartan subgroup in the theory of Galois representations attached to elliptic curves with complex multiplication.  We begin by describing the structure of $C_n$ in purely group-theoretic terms, then we proceed to give discriminant, irreducibility, and Galois-maximality results for maps of the form $k(x^2 + b)/x$.

By slight abuse of notation, we write $\iota$ for the action induced by $\iota$ on $T_n$.  Because $\iota$ acts on $T_n$ without fixed points, its action on $T_1$ is non-trivial.  
Note that for any $j<n$ there is a natural epimorphisms $C_n \to C_j$ obtained by restriction.  For a vertex $v \in T_{n}$, we define the height of $v$ to be $\min_i \{ v \in T_i \}$.

\begin{proposition} \label{autprop}
Let $\iota \in \Aut(T_n)$ be any involution whose restriction to $\Aut(T_1)$ is non-trivial.  Let $C_j$ be the centralizer in $\Aut(T_j)$ of $\iota$ restricted to $T_j$, and let $T_a$ be a subtree of $T_n$ rooted at a height-one vertex of $T_1$.  Then the map
$$h: \ker \left( C_n \to C_{1} \right) \longrightarrow \Aut(T_a)$$
given by $h(\tau) = \tau|_{T_a}$ is an isomorphism.  
\end{proposition}

\begin{proof}
Because there are exactly two branches from the root of $T_n$, there are exactly two subtrees of $T_n$ rooted at a height-one vertex of $T_1$; call them $T_a$ and $T_b$.
 The height-$n$ vertices $V$ of $T_n$ may be decomposed into the union of the height-$(n-1)$ vertices $V_a \in T_a$ and $V_b \in T_b$ .  Because $\iota$ acts non-trivially on $T_1$ and is an automorphism of $T_n$, we have $\iota(T_a) = T_b$ and $\iota(T_b) = T_a$.
 
Clearly $h$ is a homomorphism.  To show $h$ is surjective, let $\sigma \in \Aut(T_a)$ and define $\tau \in \Aut(T_n)$ by
\[
\tau |_{T_1} = \id, \quad
\tau |_{T_a} = \sigma, \quad \text{ and }
\tau |_{T_b} =\iota \sigma \iota.
\]
 One then checks that $\iota \tau \iota = \tau$.  Since $\tau$ acts trivially on $T_1$,  $\tau \in \ker(C_n\to C_1)$ and $h(\tau) = \sigma$.
 
To show that $h$ is injective, let $\tau \in \ker h$, so that $\tau(x) = x$ for all $x \in T_a$.  Then since $\tau \in C_n$ we have
$\tau(\iota(x)) = \iota(\tau(x)) = \iota(x)$; that is $\tau$ acts trivially on all elements of $T_b$ as well.  Thus $\tau = \id$.
\end{proof}

In the next corollary, we describe the kernel of the restriction map $C_n \to C_{n-1}$.  Recall that the kernel of the restriction map $\Aut(T_j) \to \Aut(T_{j-1})$ is generated by the transpositions swapping a single pair of vertices connected to a given vertex at level $j-1$, and thus is isomorphic to $(\Z/2\Z)^{2^{j-1}}$.  Recall also that the {\em Hausdorff dimension} of a subgroup $H$ of $\Aut(T)$ is defined to be
$$\lim_{n \to \infty} \frac{\log_2 \#H_n}{\log_2\#\Aut(T_n)}, \label{hausdorff}
$$
where $H_n$ is the restriction of the action of $H$ to the tree $T_n$.  This gives a rough measure of the size of $H$ in $\Aut(T)$.  

\begin{corollary}\label{autpropcor}
Assume the hypotheses of Proposition~$\ref{autprop}$, and assume also that $p_n$ is separable.  Then there is an isomorphism between $\ker(C_n \to C_{n-1})$ and $\ker(\Aut(T_{n-1}) \to \Aut(T_{n-2}))$. In particular, 
\[
	\#\ker \left( C_n \to C_{n-1} \right) = 2^{2^{n-2}}
\]
and the Hausdorff dimension of $C_\infty$ is $1/2$.
\end{corollary}

\begin{proof}
Because we have assumed that $p_n$ is separable, $T_a$ is a complete binary rooted tree of height $n-1$, and we have $\Aut(T_a) \cong \Aut(T_{n-1})$.  By Proposition~\ref{autprop} we then have a commutative diagram

\[
\begin{CD}
0 @>>> \Aut(T_{n-1}) @>>> C_n @>>> C_1 @>>> 0\\
 @V{\rm id}VV     @Vr_1VV         @Vr_2VV         @V{\rm id}VV              @.\\
0 @>>> \Aut(T_{n-2}) @>>> C_{n-1} @>>> C_1 @>>> 0\\
\end{CD}
\]

\vspace{0.1 in}
\noindent where the rows are exact and the maps $r_1$ and $r_2$ are restriction.  It is straightforward to check that this gives an exact sequence
$$\ker {\rm id} \to \ker r_1 \to \ker r_2 \to \ker {\rm id},$$
which completes the proof.  The statement about Hausdorff dimension follows since $$\#\ker \left( \Aut(T_n) \to \Aut(T_{n-1}) \right) = 2^{2^{n-1}}.\qedhere
$$
\end{proof}

The orbit of the critical point when $b=1$ will play a special role in the sequel, so we introduce the following notation.
 
\begin{definition}
Let 
\[
	\phi^*(x) = k(x^2 + 1)/x := p^*(x)/q^*(x),
\]  
and define $p_n^*$ and $q_n^*$ by the recursion in equation~\eqref{rec1}.
Finally, let $\delta_n  = kp_n^*(1)$.
\end{definition}

We now turn to the discriminant of $p_n$.  Note that one consequence of the following Corollary is that $\Disc (p_n)$ is a square in $K_1 = K(\sqrt{-b})$ for all $n \geq 2$, so that the action of $\Gal\left(K_n/K_{1}\right)$ on the roots of $p_n$ is contained in the alternating group on $2^n$ letters.  

\begin{corollary} \label{disccor}
Let $k,b  \in K^*$ and  $\phi(x) = k(x^2 + b)/x$.
Then for all $n\geq 2$, we have 
\[
	\Disc (p_n) = \pm k^{2^n(2^{n-1} - 1)} b^{2^{2n-2}} \Disc (p_{n-1})^2 p^*_n(1)^2.
\]
\end{corollary}

\begin{proof} Since $\phi(\infty) = \infty$, we have $\phi^n(\infty) = \infty \neq 0$ for all $n$.  We thus may apply~\eqref{discform2} with the following data:
\[
 d = 2, \quad  d_p = 2,  \quad   d_q=1, \quad
 \  \text{ and } \   c = k(x^2 - b), \text{ which gives } \  d_c = 2.
\]
We can then compute the exponents given in~\eqref{k1value}:
\[
	k_1 = 2^{2n-2}-2^n \quad \text{ and } \quad k_2 = 2^n.
\]
We also have 
\[
\ell(p) = \ell(c) = k, \quad \ell(q)=1, \quad
p_n = k(p_{n-1}^2 + bq_{n-1}^2), \quad q_n = p_{n-1}q_{n-1}.
\]
Since $p(x) = k(x^2+b)$ and  $q(x)=x$,
an induction shows that $p_n$ is even for all~$n$.  Thus $p_n(\sqrt{b}) = p_n(-\sqrt{b})$.  A double induction on both $q_n$ and $p_n$ gives 
\begin{equation}\label{pnrootb}
	q_n(\sqrt{b}) =b^{(2^n-1)/2}q_n^*(1) \quad \text{ and } \quad p_n(\sqrt{b}) = b^{2^{n-1}}p_n^*(1).
\end{equation}
  Finally, $\Res(q,p) = \ell(q)^2 p(0) = kb$.  The Corollary now follows from substituting the relevant values into~\eqref{discform2} and simplifying.
\end{proof}

\begin{theorem} \label{irredthm}
Let $k,b  \in K$, $\phi(x) = k(x^2 + b)/x$.  Then  $p_n$ is irreducible if none of $-b, -b\delta_i, \delta_i$ is a square in $K$ for $2 \leq i \leq n$.
\end{theorem}

\begin{remark} It is necessary to assume that both $-b\delta_i$ and $\delta_i$ are not squares in $K$.  Indeed, in the case $k=1$, $b = -5$, one has 
\[
-b\delta_2  = 25, \quad \delta_2 = -5, \quad \text{and }
p_2 = (x^2 - 5x + 5)(x^2 + 5x + 5).
\]   In the case $k = 2/3, b = 1$ one has 
\[
\delta_2  = 100/81, \quad -b\delta_2= -100/81, \quad \text{and }
p_2 = (2/27)(4x^2 + 1)(x^2 + 4).
\]
\end{remark}

\begin{proof}
We begin by considering $p_1$ and $p_2$.  Clearly $p_1$ is irreducible if and only if $-b$ is not a square in $K$.  From \eqref{rec1} we have 
\[
	p_2 = k(p_1^2 + bq_1^2) = k(p_1 - x\sqrt{-b})(p_1 + x\sqrt{-b}).
\]
  Assuming that $-b$ is not a square in $K$, we have that $p_2$ is irreducible if and only if $p_1 - x\sqrt{-b}$ is irreducible over $K(\sqrt{-b})$, which holds if $\Disc (p_1 - x\sqrt{-b}) = -b(1 + 4k^2)$ is not a square in $K(\sqrt{-b})$.  A straightforward computation shows this holds if and only if $1 + 4k^2$ is a square in $K$ or $-b$ times a square in $K$.   Since $1 + 4k^2 = k^{-1}p^*_2(1)$ and neither $-b\delta_2 = -bkp^*_2(1)$ nor $\delta_2 = kp^*_2(1)$ is a square in $K$, we conclude $p_2$ is irreducible.

Now induct on $n$.  The cases $n=1,2$ have been handled, so let $n \geq 3$ and assume that $p_{n-1}$ is irreducible.  
By Lemma \ref{irredlem}, it is enough to show that for some root~$\alpha$ of~$p_{n-1}$, 
\[
	C \left( \phi(\sqrt{b}) - \alpha \right) \left( \phi(-\sqrt{b}) - \alpha\right)\not\in K_{n-1}^{*2}.
\]
  We do this by taking the norm of the left-hand side over $K_1 = K(\sqrt{-b})$:
\begin{align} \label {tmbg2}
	N_{K_{n-1}/K_1} &\left( -C \left( \phi(\sqrt{b}) - \alpha \right) 
	\left( \phi(\sqrt{b}) + \alpha \right) \right)\nonumber\\
	& = 
	\prod_{\phi^{n-2}(\alpha) =
	 \sqrt{-b}} -C \left( \phi(\sqrt{b}) - \alpha \right) \left( \phi(\sqrt{b}) + \alpha \right).
\end{align}
Since $\phi(-x) = -\phi(x)$, $\phi^{n-2}(\alpha) = \sqrt{-b}$ implies $\phi^{n-2}(-\alpha) = -\sqrt{-b}$. 
Thus 
\[
	\left\{ \pm \alpha : \phi^{n-2}(\alpha) = \sqrt{-b}\right\}
	=\left\{\alpha : \phi^{n-1}(\alpha) = 0\right\}.
\]
Hence the right-hand side of~\eqref{tmbg2} is the same as 
\[
	 (-C)^{(\deg p_{n-1})/2} \prod_{\phi^{n-1}(\alpha) = 0} \left( \phi(\sqrt{b}) - \alpha \right).
\]
Because $\phi(\infty) = \infty$, we have $\phi^n(\infty) \neq 0$ for all $n$.  Hence $\deg p_{n-1} = 2^{n-1}$, and $(\deg p_{n-1})/2$ is even when $n \geq 3$.  Furthermore, since $\{\alpha : \phi^{n-1}(\alpha) = 0\}$ is the same as the set of roots of $p_{n-1}$ and since $p_{n-1}(\alpha) = \ell(p_{n-1}) \prod (x - \alpha)$, the left-hand side of \eqref{tmbg2} is not a square in $K_1$ provided 
\[
	\ell(p_{n-1})^{-1} p_{n-1}\left(\phi(\sqrt{b})\right) \not\in K_1^{*2}.
\]
 Finally, the recursion in~\eqref{rec1} applied in this case gives
$p_n(\sqrt{b}) = b^{2^{n-2}}p_{n-1}(\phi(\sqrt{b}))$.  Inductive arguments show that 
\[
	\ell(p_{n-1}) = k^{2^n -1} \quad\text{ and }\quad p_n(\sqrt{b}) = b^{2^{n-1}}p_n^*(1), 
\]
meaning we must show 
\[
	k^{-(2^n -1)} b^{2^{n-1} - 2^{n-2}} p^*_n(1) \not\in K_1^{*2}.
\]
  But by assumption neither $\delta_n=kp^*_n(1)$ nor 
$-b\delta_n = -bkp^*_n(1)$ is a square in $K$, and thus $\delta_n$ is not a square in $K_1$.  (To see this, suppose that $c \in K$ with $c = (a_1 + a_2 \sqrt{-b})^2$.  Then $c = a_1^2 - ba_2^2$ with either $a_1 = 0$ or $a_2 = 0$, meaning either $c$ or $-bc$ is a square in $K$.)   This completes the main induction.
\end{proof}


Recall from Corollary \ref{autpropcor} that  $[K_n : K_{n-1}] \leq 2^{2^{n-2}}$, with equality occurring if and only if $\ker(G_n \to G_{n-1}) \cong \ker(C_n \to C_{n-1})$.
Using the methods of Section~\ref{galgen}, we give a criterion ensuring that $[K_n : K_{n-1}]$ is as large as possible.  

\begin{theorem} \label{specmaxcor1}
Let $k,b  \in K^*$ and define
$\phi(x) = k(x^2 + b)/x $.
Assume that $p_{n-1}$ is irreducible and $n \geq 3$.  Then we have $[K_n : K_{n-1}] = 2^{2^{n-2}}$ provided that there exists a prime $\p$ of $K$ with 
\begin{equation}\label{valassump}
	v_\p(\delta_n) \text{ odd,} \quad v_\p(\delta_j) = 0 \text{ for }1 \leq j \leq n-1, 
	\text{  and }  v_\p(k) = v_\p(b) = v_\p(2) = 0.
\end{equation}
\end{theorem}

\begin{remark}
Because $\#C_1 = 2 = \deg p_1$ and $\#C_2 = 4 = \deg p_2$, it follows that $[K_1 : K] \leq 2$ and $[K_2 : K_1] \leq 2$.  We have $[K_1 : K] = 2$ if and only if $p_1$ is irreducible and $[K_1 : K] = 4$ if and only if $p_2$ is irreducible.  Note that $p_1$ is irreducible if and only if $-b$ is not a square in $K$, and from the proof of Theorem~\ref{irredthm} we have that $p_2$ is irreducible if and only if $-b$, $-b\delta_2$ and $\delta_2$ are all not squares in $K$.  
\end{remark}

\begin{proof}
As in the discussion preceding Lemma~\ref{stolllem}, $K_{n}$ is obtained from $K_{n-1}$ by adjoining the square roots of $\Disc\ p(x) - \beta q(x)$, as $\beta$ varies over the roots of $p_{n-1}$.  In the present case, $\Disc\ p(x) - \beta q(x) = \beta^2 - 4bk^2$.  Since $-\beta$ is also a root of $p_{n-1}$, half of the square roots are redundant, and we have
\[
	K_n = K_{n-1}\left( \sqrt{\beta^2 - 4bk^2}: p_{n-2}(\beta) = \sqrt{-b} \right).
\]

In analogy with the discussion preceding Lemma~\ref{stolllem}, we have $[K_n:K_{n-1}] = 2^{2^{n-2}}/\#V$, where 
\begin{align}
	V = &\{(e_1, \ldots, e_{2^{n-2}}) \in \mathbb{F}_2^{2^{n-2}} 
	: \prod_j (\beta_j^2 - 4bk^2)^{e_j} \in K_{n-1}^{*2} \}, \text{ and }\nonumber \\
	&\beta_1, \ldots, \beta_j \text{ are the } 2^{n-2} \text{ solutions to }
	 p_{n-2}(\beta) = \sqrt{-b}. \label{betajdef}
\end{align}

The action of $G:=\Gal\left(K_{n-1}/K(\sqrt{-b})\right)$ on the $\beta_j$ gives an action of $G$ on $V$ as linear transformations, thereby making $V$ a $\mathbb{F}_2[G]$-module.  Lemma~\ref{stolllem} now applies to show that if $\#V > 1$, then $V$ contains a $G$-invariant element.  Since $p_{n-1}$ is irreducible, $\Gal(K_{n-1}/K)$ acts transitively on the $\beta_j$.  By the definition of the $\beta_j$ in~\eqref{betajdef}, any $\sigma$ mapping one $\beta_j$ to another must fix $\sqrt{-b}$ and thus must lie in $G$.  Hence $[K_n:K_{n-1}] = 2^{2^{n-2}}$ provided that 
\[
	 \prod_j (\beta_j^2 - 4bk^2) \not\in K_{n-1}^{*2}.
\]


The hypotheses ensure that none of $b$, $k$, or $p_i^*(1)$ can be zero, which by Theorem~\ref{irredthm} shows that $p_{n-1}$ is irreducible.  Because $n \geq 3$ and $p_{n-1}$ is separable (since $K$ is perfect), there are an even number of the $\beta_j$, and we may replace $\prod_j (\beta_j^2 - 4bk^2)$ with $\prod_j (4bk^2 - \beta_j^2)$.
Further,  the roots of $p_{n-1}$ consist of $\{\pm\beta_1, \ldots, \pm \beta_j\}$, so we have that $[K_n:K_{n-1}] = 2^{2^{n-2}}$ provided that 
\begin{equation} \label{productnosquare}
\prod_{p_{n-1}(\beta) = 0} (2k\sqrt{b} - \beta) \not\in K_{n-1}^{*2}.
\end{equation}
This product is $\ell(p_{n-1})^{-1} p_{n-1}(\phi(\sqrt{b}))$, which via~\eqref{rec1} is the same as 
$\ell(p_{n-1})^{-1} p_{n}(\sqrt{b})$ up to squares, since $n \geq 3$.   Because $\ell(p_{n-1}) = k^{2^n-1}$, we have that $\ell(p_{n-1})^{-1} p_{n}(\sqrt{b})$ is a square in $K_{n-1}$ if and only if $kp_{n}(\sqrt{b})$ is a square in $K_{n-1}$.  As in equation~\eqref{pnrootb}, $p_n(\sqrt{b}) = b^{2^{n-1}}p_n^*(1)$, so~\eqref{productnosquare} holds provided that $\delta_n = kp_n^*(1) \not\in K_{n-1}^{*2}.$

By assumption in~\eqref{valassump}, there is a prime $\p$ of $K$ with $v_\p(k) = v_\p(b) = v_\p(2) = 0$.   We thus have 
\[
	 v_\p (\Disc\ p_1) =v_\p (4bk^2) = 0.
\]
 Also from~\eqref{valassump}, we assume $v_\p(\delta_j) = v_\p(kp_j^*(1)) = 0$ for $1 \leq j \leq n-1$, and since $v_\p(k) =0$, we have 
 \[
 	v_\p(p_j^*(1)) = 0 \text{ for } 1 \leq j \leq n-1.
\]
By induction,
Corollary~\ref{disccor} implies that $v_\p(\Disc\ p_{n-1}) = 0$. Therefore $\p$ does not ramify in $K_{n-1}$, whence there is a prime $\P$ of $K_{n-1}$ with $v_\P(\p)$ odd.  We then have $v_\P(\delta_n)$ odd, which means that $\delta_n$ cannot be a square in $K_{n-1}$.  
\end{proof}

By Theorem \ref{irredthm}, to show that $p_n$ is irreducible for all $n \geq 1$, it suffices to show that none of $-b, -b\delta_n,$ or $\delta_n$ is a square in $K$ for all $n \geq 2$.  By Theorem \ref{specmaxcor1}, a relatively small amount of knowledge about the primes dividing the $\delta_n$ then allow one to show $G_{n} \cong C_{n}$.   We also note that given Proposition~\ref{prop: twist prop}, the role played by $b$ is actually a minor one because it is simply a twist parameter. 

In the next section, we prove Theorem \ref{finiteindex}, which implies Theorem \ref{finiteindex1}.  We then give several sufficient conditions to show $- b\delta_n$ and $\delta_n$ are not squares in $K$ (Theorems \ref{nosquare}--\ref{fixedpoint}), before giving a further sufficient condition on $k$ that ensures that $G_n \cong C_n$ (Theorem \ref{maxthm}).


%

\section{Maximality and finite index results}
In this section we apply the results of Section \ref{galaut} to obtain results showing $G_\infty$ is a large subgroup of $C_\infty$ in many cases.  The map $\phi(x) = k(x^2 + b)/x \in K(x)$ is given in homogeneous coordinates by  
\begin{equation} \label{cmphi}
\phi([X,Y]) = [k(X^2 + b Y^2), XY], \qquad k, b \in K.
\end{equation}
Recall that 
\begin{align*}
	P_n(X,Y)  &= 
	\begin{cases}
	k\left(X^2 + bY^2\right) & \text{ if } n=1, \\
	k\left(P_{n-1}(X,Y)^2 + bQ_{n-1}(X,Y)^2\right) & \text{ if } n\geq 2, \text{ and }
	\end{cases}\\
	Q_n(X,Y) & = \begin{cases}
	XY & \text{ if } n=1, \\
	 P_{n-1}(X,Y)Q_{n-1}(X,Y) & \text{ if } n\geq 2.
	 \end{cases}\\
\end{align*}
Thus $\phi^n([X,Y]) = [P_n(X,Y), Q_n(X,Y)]$.  Recall also that 
\begin{align*}
	p_n(x) &= P_n(x,1) & q_n(x) &= Q_n(x,1),\\
	P_n^*(X,Y) &= P_n(X,Y)|_{b=1}, & Q_n^*(X,Y) &= Q_n(X,Y)|_{b=1},\\
	p_n^*(x) &= P_n^*(x,1),  & q_n^*(x) &= Q_n^*(x,1),\\
	\phi^*(x) & = k(x^2+1)/x, \text{ and } & \delta_n & = kp_n^*(1).
\end{align*}

Many of our results in this section exclude the case where $\phi$ is post-critically finite, and so we begin by showing that this case is rare.  

\begin{proposition}\label{pcf}
Let $K$ be a number field and $\phi(x) = k(x^2+b)/x$ with $k \in K^*$.  If $\phi$ is post-critically finite, then the standard \textup(absolute\textup) multiplicative height of $k$ is at most $2$.  In particular, there are only finitely many post-critically finite $\phi$ over any number field.  
\end{proposition}

\begin{remark} This is best possible, since $k = \pm 1/2$ give  post-critically finite maps.  Using Proposition \ref{pcf}, one easily checks that these are the only $k \in \Q$ that give post-critically finite maps.  
\end{remark}

\begin{proof}  
Begin by noting that $\phi$ is conjugate to $k(x^2+1)/x$ over $\overline{K}$, and a map is post-critically finite if and only if all its conjugates are.   Therefore we may consider $\phi^*(x) = k(x^2+1)/x$.
The critical points of $\phi^*$ are $\pm 1$, and their orbits are interchanged by the involution $z \mapsto -z$.  So $\phi^*$ is post-critically finite if and only if the orbit of $z=1$ is finite.

Recall that the standard (absolute) multiplicative height of $k \in K$ is defined to be 
$$\left( \prod_{v \in M_K} \max\{1, |k|_v\}^{n_v} \right)^{1/[K:\Q]},$$
where $M_K$ is the set of places of $K$ and $n_v = [K_v : \Q_v]$ is the local degree of $v$.  Consider first an archimedean place $v$ of $K$, and for simplicity denote $|\cdot|_v$ by $|\cdot|$.  Suppose $|k| > 1$.  One checks that $\infty$ is a fixed point of $K$ with multiplier $1/k$, and hence is attracting.  By \cite[Thm 9.3.1]{beardon}, $\infty$ must attract a critical point monotonically (i.e. the critical point does not land on $\infty$), proving that $\phi$ is not post-critically finite.  

Now let $v$ be a non-archimedean place of $K$.  If $|k| > 1$, then for $|x| > 1$ we get
\begin{align*}
\left|\phi(x)\right| & = 
|k| \cdot \left|\frac{x^2+1}{x}\right| 
 = |k| \cdot  \frac{|x^2|}{|x|}  > |x|.
\end{align*}
It follows that $x$ cannot have a finite orbit under $\phi$.   Now $|\phi(1)| = |2k|$, and this equals $|k|$ provided that $v$ does not lie over $2$. If $v$ does divide $2$, then we still have $|2k| > 1$ provided that $|k| > |1/2| = 2$.  Hence if $\phi$ is post-critically finite, we must have $|k|_v \leq 1$ for each place $v \in M_K$ not over~$2$, and $|k|_v \leq 2$ for each place $v \in M_K$ over $2$.

Suppose that $\phi$ is post-critically finite, and assume $[K:\Q]=d$.  We have
\begin{equation*}
\prod_{v \in M_K} \max\{1, |k|_v\}^{n_v}\\
\leq 
\prod_{v\mid 2} 2 ^{n_v} = 2^d.
\end{equation*}
Taking the $d^{\text{th}}$ root gives the desired height bound of $2$.  
\end{proof}


\begin{lemma} \label{relprime}
Suppose that for some prime $\p$ of $\calO_K$, $v_\p(\delta_n) > 0$ and $v_\p(\delta_m) > 0$ for some $m \neq n$.  Then $v_\p(k) > 0$.  
\end{lemma}

\begin{proof}
Without loss of generality, we may assume $n < m$ and that $n$ is the smallest positive integer satisfying $v_\p(\delta_n) > 0$.  The rough idea is that $\phi$ maps $0$ to $\infty$, which is a fixed point.  Thus if $\phi^n(1) \equiv 0 \bmod{\p}$, then $\phi^m(1) \not\equiv 0 \bmod{\p}$, and $\p$ cannot divide both $p_n(1)$ and $p_m(1)$.  

We are given that $v_\p(\delta_n) = v_\p(kp^*_n(1)) > 0$.  From the recursion, we see that $p^*_n(1)$ is a polynomial in $k$ with integral coefficients and no constant term;  hence $v_\p(k) < 0$ implies $v_\p(p^*_n(1)) < 0$.  It follows that either $v_\p(k) > 0$ and we're done or $v_\p(k) = 0$ and $v_\p(p^*_n(1)) > 0$.  We assume the latter scenario and derive a contradiction.  

We have $v_\p(P^*_n(1,1)) > 0$, and thus 
$Q^*_{n+1}(1,1) = P^*_n(1,1)Q^*_n(1,1) \equiv 0 \bmod{\p}$ and 
$$
P^*_{n+1}(1,1) = k\left(P^*_n(1,1)^2 + Q^*_n(1,1)^2\right) 
\equiv kQ^*_n(1,1)^2 \bmod{\p}.
$$ 
Induction gives $Q^*_m(1,1) \equiv 0 \bmod{\p}$ and 
$$
P^*_m(1,1) \equiv k^{2^{m-n}-1}Q^*_n(1,1)^{2^{m-n}} \bmod{\p}.
$$  
Now $Q^*_n(1,1) = P^*_1(1,1)P^*_2(1,1) \cdots P^*_{n-1}(1,1)$, so by the minimality of $n$ we have 
$Q^*_n(1,1) \not\equiv 0 \bmod{\p}$.  By assumption we have $k \not\equiv 0 \bmod{\p}$, and so it follows that $p^*_m(1) = P^*_m(1,1) \not\equiv 0 \bmod{\p}$.  This contradicts our supposition that $v_\p(\delta_m) > 0$.
\end{proof}

\begin{theorem} \label{finiteindex}
Let $\phi$ be defined as in \eqref{cmphi}.   Suppose that none of $-b$, $\delta_n$, and $-b\delta_n$ is a square in $K$ for $n \geq 2$, and also assume that $\phi$ is not post-critically finite.  Then $G_\infty$ has finite index in $C_{\infty}$.  
\end{theorem}

\begin{remark}
From the proof below, it follows that only finitely many  of the numbers $\delta_n$ and $-b \delta_n$ can be squares.  However, this is not enough to ensure that $p_n(x)$ is irreducible for all $n$, and failure of this irreducibility provides an obstacle to showing $[C_{\infty} : G_{\infty}] < \infty$.  
\end{remark}

\begin{proof}
We first claim that for any $c \in K^*$, $c\delta_n$ is a square in $K$ for at most finitely many~$n$.  Note that $c\delta_n = ckp_n^*(1)$.  Let $\phi^*(x) = k(x^2 + 1)/x$  and apply equation \eqref{rec2}  to see that $ckp_n^*(1)$ is a square if and only if 
\[
	ck \cdot p^*\left(\phi^{*\, {n-1}}(1)\right) \in K^{*2}.
\]
 Without loss of generality, we may take $n\geq 4$, and thus rewrite
  \[
	ck \cdot p^*\left(\phi^{*\, {n-1}}(1)\right) =   
	ck \cdot p^*\left(\phi^{*2}\circ \phi^{*\, n-3}(1)\right).
\]
%
 It follows that $c\delta_n$ is a square if and only if the curve $ck \cdot p^*(\phi^{*2}(x)) = y^2$ has a $K$-rational point $(x,y)$ with $x = (\phi^*)^{n-3}(1)$.  Since 
\[
	p^*(\phi^{*2}(x)) = k  \left(p_2^*(x)^2 + q_2^*(x)^2\right) / q_2^*(x)^2 ,
\]
 this is equivalent to the curve
\begin{equation} \label{curve}
C: y^2 = c \left(p_2^*(x)^2 + q_2^*(x)^2 \right)
\end{equation}
having a rational point with $x = (\phi^*)^{n-3}(1)$.  The right-hand side of \eqref{curve} is simply $c k^{-1}p_3^*(x)$, whose discriminant, by Corollary \ref{disccor}, is divisible only by $c$, $k$, $b$, $p_1^*(1), p_2^*(1),$ and $p_3^*(1)$.   None of these is zero because none of $\delta_1, \delta_2, $ and $\delta_3$ is a square in $K$ by hypothesis.  
Thus the right-hand side of \eqref{curve} has distinct roots, and hence the genus of $C$ is $3$.  It follows from Faltings' Theorem \cite[Part E]{jhsdioph} that $C$ has only finitely many rational points, and thus $c\delta_n$ is a square for only finitely many $n$.  


To prove the Theorem, note that the hypotheses on $-b$, $-b\delta_n$ and $\delta_n $ allow us to apply Theorem~\ref{irredthm} to show that $p_n(x)$ is irreducible for all $n$.  We now wish to apply Theorem ~\ref{specmaxcor1}.  
Let $S$ be a finite set of places of $K$, including all places dividing $k,b,$ or $2$, and all archimedean places.  Expand $S$ further, if necessary, so that the ring $\calO_{K,S}$ of $S$-integers is a principal ideal domain.  Let $U_{K,S}$ denote the multiplicative group of $S$-units.  Note that since $\delta_n \in \Z[k^2]$, there exists $a_n \in K$ with $a_n^2 \delta_n \in \calO_K$.  Since $\calO_{K,S}$ is a UFD, we may write for each $n$, 
\begin{equation} \label{dec}
a_n^2 \delta_n = u \beta^2 \prod_{i = 1}^{j_n} \pi_i,
\end{equation}
with $u \in U_{K,S}/U_{K,S}^2$, $\beta \in \calO_{K,S}$, and $\pi_i \in \calO_{K,S}$ irreducible, and this decomposition is unique.  We permit the product on the right of \eqref{dec} to be empty, and say $j_n = 0$ in this case.   

By Dirichlet's Theorem for $S$-units \cite[p. 174]{frotay}, $U_{K,S}/U_{K,S}^2$ is a finite group, and we let $\Sigma$ consist of a set of coset representatives.  
Suppose that there are infinitely many $n$ for which $j_n = 0$.  Because $\Sigma$ is finite, there is a product $c$ of elements in $\Sigma$ with $ca_n^2\delta_n$ a square in $\calO_{K,S}$ for infinitely many $n$.  This contradicts the conclusion of the previous paragraph.  

It follows that for all but finitely many $n$, there must be at least one $\pi_i$ in \eqref{dec}.  Thus setting $\p_i = (\pi_i) \cap \calO_K$ we have 
\[
v_{\p_i}(\delta_n) \text{ odd, and } 
v_{\p_i}(k) = v_{\p_i}(b) = v_{\p_i}(2) = 0.
\]
  Because
$v_{\p_i}(k) = 0$, Lemma \ref{relprime} implies that $\p_i$ divides at most one $\delta_n$.  Theorem \ref{specmaxcor1} then applies to complete the proof. 
\end{proof}


 In light of Theorem \ref{finiteindex}, we now study the quantities $-b$, $\delta_n$, and $-b\delta_n$.  We begin with a fundamental result on the polynomials $P_n(X,Y)$ and $Q_n(X,Y)$.



\begin{lemma}\label{an&bn}
Let $S_n, T_n \in \Z[k,X,Y]$ be the polynomials not divisible by $k$ that satisfy $P_n = k^{s(n)}S_n$ and 
$Q_n = k^{t(n)}T_n$ respectively, for some $s(n), t(n) \in \Z$.  Then we have 
\begin{align*}S_n
		&= \begin{cases}
			S_{n-1}^2 + bT_{n-1}^2 & \text{if $n$ is odd}\\
			k^2S_{n-1}^2 + bT_{n-1}^2 & \text{if $n$ is even},\\
		\end{cases}\\
T_n
		&= S_{n-1}T_{n-1}, \\
	s(n) 
		&= \frac{1}{3}\left(2^n-(-1)^n\right), \qquad \text{and}\\
	t(n)
		&=
			\begin{cases}
				s(n)-1 &\text {if $n$ is odd}\\
				s(n) &\text{if $n$ is even.}
			\end{cases}
\end{align*}
Moreover, for any $n$, $S_n$ and $T_n$ are homogeneous in $X$ and $Y$, and relatively prime as polynomials in $X$ and $Y$ with coefficients in $\Z[k]$.
\end{lemma}

\begin{proof}
We proceed by induction.  In this case it will be convenient for us to start with $n = 0$, in which case we put $\phi^0([X,Y]) = [X,Y]$.  This gives 
\[
	S_0 = X, \quad T_0 = Y, \text{ and } \quad s(0) = t(0) = 0.
\]  
For $n = 1$, we have $\phi([X,Y]) = [k(X^2+bY^2), XY]$, showing that 
\[
	S_1 = S_0^2 + bT_0^2, \quad T_1 = S_0T_0, \quad s(1) = 1,  \text{ and } \quad t(1) = 0,
\]
 which agree with the statements in the Lemma.  

For $n=2$, we have
\begin{align*}
	P_2(X,Y) 
		&= k\left(k^2(X^2+bY^2)^2+b(XY)^2\right) \text{ and }\\
	Q_2(X,Y)
		&= k(X^2+bY^2)XY,
\end{align*}
so that 
\[
	S_2 = k^2S_1^2 + bT_1^2, 
	\quad 
	T_2 = S_1T_1, \text{ and } 
	\quad s(2) = t(2) = 1,
\]
 which again agree with the statements in the Lemma.  

Now suppose that $n$ is even and that the statement of the Lemma holds for $n-1$, so in particular $t(n-1)=s(n-1) - 1$.  Then
\begin{align*}
	P_{n} 
		= k\left( P_{n-1}^2+ b Q_{n-1}^2\right) &=k\left(k^{2s(n-1)}S_{n-1}^2 + bk^{2t(n-1)}T_{n-1}^2\right) \\
		&=k \cdot k^{2s(n-1)-2}\left(k^2S_{n-1}^2 + bT_{n-1}^2\right).
\end{align*}
Since $k$ does not divide $T_{n-1}$, it also does not divide $k^2S_{n-1}^2 + bT_{n-1}^2$, showing that 
$S_n = k^2S_{n-1}^2 + bT_{n-1}^2$.  Thus 
\begin{align*}
s(n) &= 2s(n-1)-1  = \frac{2}{3}\left(2^{n-1}-(-1)^{n-1}\right) - 1 &\text{ (by induction)}\\
 & = \frac{1}{3}\left(2^{n}+2-3\right) &\text{ (since $n-1$ is odd)}\\
 &= \frac{1}{3}\left(2^{n} - (-1)^n\right).
\end{align*}

We also have
$$
	Q_{n}
		= P_{n-1}Q_{n-1}
		= \left(k^{s({n-1})} S_{n-1} \right) 
			\left(k^{t(n-1)}T_{n-1}\right)
		= k^{2s(n-1) - 1} S_{n-1}T_{n-1}.
$$
Since $k$ divides neither $S_{n-1}$ nor $T_{n-1}$, it also does not divide $S_{n-1}T_{n-1}$, showing that $T_n = S_{n-1}T_{n-1}$.  Thus $t(n) = 2s(n-1) - 1 = s(n)$. 		

Now suppose that $n$ is odd and that the statement of the Lemma holds for $n-1$, so in particular $t(n-1)=s(n-1)$.  Then
\begin{align*}
	P_{n} 
		= k\left(P_{n-1}^2+ bQ_{n-1}^2\right)
		&=k\left(k^{2s(n-1)}S_{n-1}^2 + bk^{2t(n-1)}T_{n-1}^2\right)\\
		&=k \cdot k^{2s(n-1)}\left(S_{n-1}^2 + bT_{n-1}^2\right).
\end{align*}
As in the case of $n$ even, we have $S_n = k^2S_{n-1}^2 + bT_{n-1}^2$.  Thus 
\begin{align*}
s(n) &= 2s(n-1)+1 = \frac{2}{3}\left(2^{n-1}-(-1)^{n-1}\right) + 1 & \text{ (by induction)}\\
&=\frac{1}{3}\left(2^{n}-2+3\right) & \text{ (since $n-1$ is even)}\\
&= \frac{1}{3}\left(2^{n} - (-1)^n\right).
\end{align*}

We also have
$$
	Q_{n}
		= P_{n-1}Q_{n-1}
		= \left(k^{s({n-1})} S_{n-1} \right) 
			\left(k^{t(n-1)}T_{n-1}\right)
		= k^{2s(n-1)} S_{n-1}T_{n-1}.
$$
As in the case of $n$ even, it follows that $T_n = S_{n-1}T_{n-1}$.  Thus $t(n) = 2s(n-1) = s(n) - 1$.    
 	
It remains to show that $S_n$ and $T_n$ are relatively prime as homogeneous polynomials in $X$ and $Y$ with coefficients in $\Z[k]$.  Assume inductively the same statements hold for $S_{n-1}$ and $T_{n-1}$.  The homogeneity of $S_n$ and $T_n$ follows immediately from the recursions in the Lemma, which have already been established.  Let $F$ be an irreducible non-constant homogeneous polynomial in $X$ and $Y$ with coefficients in $\Z[k]$.  If $F$ divides $T_n$, then $F$ must divide either $S_{n-1}$ or $T_{n-1}$, but cannot divide both since $S_{n-1}$ and $T_{n-1}$ are relatively prime.  From the formula for $S_n$ in the Lemma it follows that $F$ cannot divide $S_n$, regardless of the parity of $n$.    
\end{proof}

For the remainder of this section we assume that $b=1$, as in Conjecture~\ref{mainconj}.  We thus have $\phi^* = \phi$, $p_n^* = p_n$, and $q_n^* = q_n$. 
We put Lemma \ref{an&bn} to use to study the $\delta_n$, which will allow us to apply Theorem \ref{finiteindex} in the case where $K$ is real, since $\delta_n > 0$ and thus $-\delta_n$ cannot be a square.  Before proceeding, we note that 
\[
\delta_n = kP_n(1,1) = k^{s(n)+1}S_n(1,1)
\]
 by Lemma \ref{an&bn}.  Since $s(n) + 1$ is always even, we have that $\delta_n$ is a square in $K$ if and only if $S_n(1,1)$ is a square in $K$.  To make this a bit more concrete, here are the first few $P_n(1,1)$ and $Q_n(1,1)$, with corresponding $S_n(1,1)$ and $T_n(1,1)$ easy to read off.  
\begin{align*}
P_1(1,1)  & =  2k,  & P_2(1,1) & =  k(4k^2 + 1), & P_3(1,1) & = k^3(16k^4 + 8k^2 + 5), \\
Q_1(1,1) &  =  1,    & Q_2(1,1) &  =  2k, & Q_3(1,1), &= k^2(8k^2 + 2), 
\end{align*}
\begin{align*}
P_4(1,1) & = k^5(256k^{10} + 256k^8 + 224k^6 + 144k^4 + 57k^2 + 4), \text{ and} \\
Q_4(1,1) & = k^5(128k^6 + 96k^4 + 56k^2 + 10).
\end{align*}

\begin{lemma} \label{five}
Suppose that $b  = 1$ and there is a prime $\p$ of $\calO_K$ with $\p \mid (5)$ and $[\calO_K/\p : \Z/5\Z]$ odd.  If $k \equiv \pm 2 \bmod{\p}$, then neither of $\pm \delta_n$ is a square for any $n$.  
\end{lemma}

\begin{proof}
From Lemma \ref{an&bn} and the fact that $\calO_K/\p$ has characteristic 5, we have that the sequence $(S_n(1,1), T_n(1,1))$ modulo $\p$ is $(2,1), (2,2), (3,4), (2,2), (3,4), \ldots$ and repeats in the obvious way.  Because $[\calO_K/\p : \Z/5\Z]$ is odd, $\calO_K/\p$ has no quadratic sub-extensions, and thus neither of $\pm 2$ is a square in $\calO_K/\p$.  Hence $\pm S_n(1,1)$ is not a square modulo $\p$ for all $n$. 
\end{proof}

\begin{corollary} Suppose that $K$ is a number field of odd degree, and take $b = 1$.  Then there is a congruence class of $k \in K$ with $[C_\infty : G_\infty]$ finite.  
\end{corollary}

\begin{proof} 
Because $K$ has odd degree, it has no quadratic sub-extenstions, and hence $-1$ cannot be a square in $K$.  Moreover, the product of the residue class degrees of ideals of $\calO_K$ dividing $(5)$ must be odd, and thus there is some $\p \mid (5)$ with $[\calO_K/\p : \Z/5\Z]$ odd.  By Lemma~\ref{five}, when $k \equiv \pm 2 \bmod{\p}$, neither $\pm \delta_n$ is a square for any $n$.  By Theorem~\ref{finiteindex}, 
$[C_{\infty} : G_{\infty}]$ is finite for all $k \equiv \pm 2 \bmod{\p}$.  
\end{proof}

We now present several results that show $\delta_n$ is not a square for all $n$ provided that there exist certain primes of $\calO_K$ and $k$ satisfies conditions relating to these primes.  These lead into Corollary \ref{irredfam}, which shows that Conjecture~\ref{mainconj} is true for certain real number fields $K$, including $\Q$.  


\begin{theorem} \label{nosquare}
Suppose that $b = 1$ and $v_\p(k) = 0$ for some prime $\p \subset \calO_K$ with $\#\calO_K/\p = 2$.  Then $\delta_n$ is not a square for all $n \geq 2$.  
\end{theorem}

\begin{proof}
Note first that $\#\calO_K/\p = 2$ implies $\p \mid (2)$, and let $e \geq 1$ be such that $\p^e || (2)$.  We claim that $\p \nmid \left(P_n(1,1)\right) $ and $\p^e || \left( Q_n(1,1)\right) $ for all $n \geq 2$. We have 
\[
\left( P_1(1,1) \right) = (2k) \text{ and } \left( Q_1(1,1) \right) = (1), \text{ whence }
\left( P_2(1,1) \right) = \left( k(4k^2 + 1)\right), 
\]
which is not divisible by $\p$ because $v_\p(k) = 0$ and $\p \mid (2)$.  Also, $(Q_2(1,1)) = (2k)$, which is exactly divisible by $\p^e$.  Since 
\[
	P_n(1,1) = k(P_{n-1}(1,1)^2 + Q_{n-1}(1,1)^2)
	\text{ and }Q_n(1,1) = P_{n-1}(1,1) Q_{n-1}(1,1),
\]
 the claim follows by induction.

Now suppose that $\delta_n \in K^2$ for some $n \geq 2$.  Then $kP_n(1,1) \in K^2$, and so
\[
	P_{n-1}(1,1)^2 + Q_{n-1}(1,1)^2 = z^2 \text{ for some } z \in K.
\]
  Rewrite this as 
  \[
  	Q_{n-1}(1,1)^2 = (z+P_{n-1}(1,1))(z-P_{n-1}(1,1)),
\]
 and to ease notation set 
 \[
 	z+P_{n-1}(1,1) = s \quad \text{ and }\quad  z - P_{n-1}(1,1) = t.
\]
  This gives 
\[
	Q_{n-1}(1,1)^2 = st \quad \text{ and } \quad 2P_{n-1}(1,1) = s-t.
\]
  From the previous paragraph, $v_\p(Q_{n-1}(1,1)) = e = v_\p(2)$ and $v_\p(P_{n-1}(1,1)) = 0$, giving
\begin{align}
2e &= v_\p(s) + v_\p(t) \text{ and } \label{ifequal}\\
e &= v_\p(s-t); \text{ hence}\label{diffvals}\\
 2v_\p(s-t) &= v_\p(s) + v_\p(t). \label{valuations}
\end{align}

If $v_\p(s) \neq v_\p(t)$, then $2v_\p(s-t) = 2 \max{\{v_\p(s), v_\p(t)\}} > v_\p(s) + v_\p(t)$, contradicting~\eqref{valuations}.  If $v_\p(s) = v_\p(t)$, they are both $e$ by equation~\eqref{ifequal}. So $s, t \in \p^e$ and $s, t \not\in \p^{e+1}$, so  neither is the identity in $\p^e/\p^{e+1}$.  But $\p^e/\p^{e+1} \cong O_K/\p$ (see e.g. \cite[p. 43]{frotay}) and thus has only two elements, implying that $s-t \in \p^{e+1}$.  This contradicts \eqref{diffvals}, proving the theorem.  
\end{proof}

\begin{theorem} \label{three}
Suppose that $b = 1$ and there is $\p \subset \calO_K$ with $\#\calO_K/\p = 3$.  Then $\delta_n$ is not a square for all odd $n$.  If $v_\p(k) = 0$ then $\delta_n$ is not a square for all $n \geq 1$.  
\end{theorem}

\begin{proof} Note that $\calO_K/\p \cong \Z/3\Z$.  Because $S_1(1,1) = 2$ and $T_1(1,1) = 1$, and the sum of the squares of two non-zero elements of $\Z/3\Z$ cannot be zero, it follows by induction that $S_n(1,1)$ and $T_n(1,1)$ are not zero in $\calO_K/\p$.  It then follows immediately from the recurrence for $S_n$ in Lemma \ref{an&bn} that $S_n(1,1) \equiv 2 \bmod{\p}$ for $n$ odd, and thus cannot be a square in $K$ for $n$ odd.  If $v_\p(k) = 0$, then the same statement holds for $n$ even.  
\end{proof}

\begin{theorem} \label{fiveseven} 
Let $b = 1$.   If one of the following holds then $\delta_n$ is not a square for all $n \geq 1$:
\begin{enumerate}
\item There is $\p \subset \calO_K$ with $\calO_K/\p \cong \Z/7\Z$ and $k \equiv 2, 5 \bmod{\p}$.
\item There is $\p \subset \calO_K$ with $\calO_K/\p \cong \Z/7\Z$, $k \equiv 1, 6 \bmod{\p}$, and the hypotheses of Theorem~$\ref{three}$ hold.
\end{enumerate} 
\end{theorem}

\begin{proof}
For $k \equiv \pm 2 \bmod{\p}$, the sequence $(S_n(1,1), T_n(1,1))$ modulo $\p$ is $(2, 1), (3, 2)$ and then a repeating cycle of $(6, 6), (5, 1), (5, 5), (6, 4), (3, 3), (3, 2)$, so $S_n(1,1)$ is never a square modulo $\p$.  For $k \equiv \pm 1 \bmod{\p}$, the sequence in question consists of the length-12 repeating cycle 
\[
(2, 1), (5, 2), (1, 3), (3, 3), (4, 2), (6, 1), (2, 6), (5, 5), (1, 4), (3, 4), (4, 5), (6, 6).
\]
  Hence $S_n(1,1)$ is not a square modulo $\p$ for all even $n$, and combined with Theorem~\ref{three} this shows that $\delta_n$ is not a square for all $n \geq 1$.  
\end{proof}

\begin{theorem} \label{fixedpoint}
Let $b = 1$.  Suppose that one of the following holds:
\begin{enumerate}
\item $v_\p(2k - 1) > 0$ or $v_\p(2k + 1) > 0$ for a prime $\p \subset \calO_K$ such that $2$ is not a square in $\calO_K/\p$.  
\item $v_\p(2k^2 - k + 1) > 0$ for a prime $\p \subset \calO_K$ such that $-k$ is not a square in $\calO_K/\p$
\item $v_\p(2k^2 + k + 1) > 0$ for a prime $\p \subset \calO_K$ such that $k$ is not a square in $\calO_K/\p$
\end{enumerate}
Then $\delta_n$ is not a square for all $n$.      
\end{theorem}

\begin{remark}
When $K = \Q$, case (1) of Theorem \ref{fixedpoint} applies provided that there is a prime $p$ with $v_p(2k \pm 1) > 0$ and $p$ congruent to $3$ or $5$ modulo $8$.  In particular, if $k$ is an integer and $2k \not\equiv 0 \bmod{8}$, then one of $2k \pm 1$ must be equivalent to $3$ or $5$ modulo $8$, implying that some divisor of $2k \pm 1$ is equivalent to $3$ or $5$ modulo $8$.  Hence if $k$ is an integer not divisible by 4 then $\delta_n$ is not a square for all $n$.
\end{remark}

\begin{proof}
It follows from the definitions of $P_n(X,Y)$ and $Q_n(X,Y)$ that if $\epsilon_n = kQ_n(1,1)$, then
\begin{equation} \label{delta}
\delta_n = \delta_{n-1}^2 + \epsilon_{n-1}^2 \qquad \epsilon_n = (1/k)\delta_{n-1}\epsilon_{n-1}.
\end{equation}

Suppose first that we are in case (1).  
One checks that $\delta_2 - \delta_1$, $\delta_3 - \delta_2$, and $\epsilon_3 - \epsilon_1$ are all divisible by both $2k - 1$ and $2k + 1$.  Hence for $\p \mid (2k-1)$ or $\p \mid (2k + 1)$, we have 
$\delta_3 \equiv \delta_2 \equiv \delta_1 \bmod{\p}$ and $\epsilon_3 \equiv \epsilon_1 \bmod{\p}$, which by \eqref{delta} and induction ensures that $\delta_n \equiv \delta_1 \bmod{\p}$ for all $n$.  But $\delta_1 = 2k^2$, which is not a square modulo $\p$ by assumption.

For the other cases, one checks that $2k^2 \pm k + 1$ divides 
$\delta_3 - \delta_2$, $\delta_4 - \delta_3$, and $\epsilon_4 - \epsilon_2$.   As in the previous paragraph, this implies that $\delta_n \equiv \delta_2 \bmod{\p}$ for all $n \geq 2$.  Now $\delta_2 = k^2(4k^2 + 1)$.  Suppose first that $\p \mid (2k^2 - k + 1)$.  Then $4k^2 \equiv 2k - 2 \bmod{\p}$, so $\delta_2$ is a square modulo $\p$ if and only if $2k - 1$ is.  But $-2k^2 + k \equiv 1 \bmod{\p}$, so 
$-k(2k - 1)$ is a square modulo $\p$.  Hence $2k-1$ is a square modulo $\p$ if and only if $-k$ is.  

If $\p \mid (2k^2 + k + 1)$, then $4k^2 \equiv -2k - 2 \bmod{p}$, so $\delta_2$ is a square modulo $p$ if and only if $-2k - 1$ is.  But $-2k^2 - k \equiv 1 \bmod{p}$, so 
$k(-2k - 1)$ is a square modulo $p$.  Hence $-2k-1$ is a square modulo $p$ if and only if $k$ is.  
\end{proof}

\begin{corollary} \label{irredfam}
Let $K$ be a number field with a real embedding, and suppose $b = 1$ and the hypotheses of one of Theorems~$\ref{three}$--$\ref{fixedpoint}$ hold.  Then none of $\{\pm \delta_n : n = 2, 3, \ldots\}$ is a square in $K$, $p_n(x)$ is irreducible for all $n \geq 1$ and $G_{\infty}$ has finite index in $C_{\infty}$.
\end{corollary}

\begin{proof}
Because $K$ has a real embedding, $-1$ is not a square in $K$, and by hypothesis we have that $\delta_n$ is not a square for all $n \geq 2$.  Moreover, 
\[
\delta_n = kP_n(1,1) = k^2(P_{n-1}(1,1)^2 + Q_{n-1}(1,1)^2) > 0  \text{ for all } n.
\]
  Hence $-\delta_n$ cannot be a square in $K$ either.  Thus by Theorem \ref{irredthm}, $p_n(x)$ is irreducible for all $n \geq 1$.  Theorem \ref{finiteindex} applies as well, proving the Corollary.  
\end{proof}

\begin{remark} In the case $K = \Q$, Corollary \ref{irredfam} applies to most values of $k$.  For instance, when $k$ is a positive integer, Theorem \ref{fixedpoint} alone applies to all $k \leq 10000$ except for $55$ values of $k$.  Of these, the third part of Theorem \ref{fiveseven} eliminates $21$ values.  

The remaining 34 values can be ruled out with additional congruences.  As in the argument of Theorem \ref{fiveseven}, one can compute the eventually periodic sequence of ordered pairs $(\delta_n \bmod{p}, \epsilon_n \bmod{p})$ and check whether $\delta_n$ is a square mod $p$ for elements in the cycle and also in the tail, i.e., those elements before the cycle begins.  For instance, when $p = 11$ and $k \equiv \pm 1 \bmod{11}$, the sequence has a tail of length 2 and a cycle of length 4.  The four elements of the cycle have $\delta_n$ not a square, and while the second element of the tail is a square, we need not worry about $\delta_2$ being a square, since $4k^2 + 1$ is not a square for $k$ a positive integer. Hence $\delta_n$ is not a square for all $n$ when $k \equiv \pm 1 \bmod{11}$.  This eliminates 11 of the 34 remaining $k$ values.  

The 23 values of $k$ that still remain may be handled with congruences involving higher moduli, as given in Table \ref{higherk}.  The second column is the modulus $p$ of the congruence, the third and fourth columns give the tail length and cycle length of the sequence $(\delta_n \bmod{p}, \epsilon_n \bmod{p})$, and the fifth column gives the $n$, if any, such that $\delta_n$ is a square modulo $p$.  When there is a prime $p < 200$ such that both the tail and cycle contain no $\delta_n$ that are squares modulo $p$, we have listed that prime.  Otherwise, we have chosen $p$ to minimize the exceptional $n$, which are always at most 3.  Note that by Theorem \ref{three}, $\delta_1$ and $\delta_3$ are not squares, and as noted in the previous paragraph $\delta_2$ cannot be a square for $k$ a positive integer.  We remark that the $\delta_n$ need not be distinct for different values of $n$, which explains why it is reasonable to have long cycles not containing square $\delta_n$, as in the case of $k = 840, p = 197$ or $k = 1620, p = 37$.  

\begin{table}
\begin{tabular}{|c|c|c|c|c|}
\hline
$k$ & $p$ & tail length & cycle length & exceptional $n$\\
\hline
444 & $61$ & $0$ & $4$ &   \\
\hline
840 & $197$ & $1$ & $84$ & \\
\hline
1620 & $37$ & $0$ & $36$ &  \\
\hline
1764 & $83$ & $0$ & $60$ &  \\
\hline
3000 & $13$ & $1$ & $12$ &  \\
\hline
3336 & $37$ & $2$ & $6$ & $n = 2$ \\
\hline
4176 & $13$ & $1$ & $12$ & \\
\hline
4224 & $19$ & $0$ & $6$ &  \\
\hline
4620 & $41$ & $4$ & $4$ & $n = 1,2$ \\
\hline
4704 & $43$ & $2$ & $6$ &  \\
\hline
5184 & $13$ & $1$ & $12$ &  \\
\hline
5904 & $31$ & $3$ & $4$ & $n = 1,2$ \\
\hline
6240 & $17$ & $4$ & $4$ & $n = 1$ \\
\hline
6384 & $37$ & $4$ & $2$ & $n = 2,3$ \\
\hline
6996 &$71$ & $2$ & $4$ & $n = 1$ \\
\hline
7224 & $17$ & $4$ & $4$ & $n = 1$ \\
\hline
7620 & $31$ & $2$ & $4$ & $n = 1$ \\
\hline
7836 & $13$ & $1$ & $12$ &  \\
\hline
7956 & $83$ & $1$ & $60$ &  \\
\hline
8004 & $31$ & $2$ & $4$ & $n=1$ \\
\hline
8316 & $19$ & $0$ & $6$ &  \\
\hline
9720 & $131$ & $3$ & $12$ &  \\
\hline
9804 & $29$ & $1$ & $12$ &  \\
\hline
\end{tabular}

\bigskip

\caption{Congruences used to show $[G_\infty : C_\infty]$ is finite for given values of $k$.}  \label{higherk}
\end{table}

\end{remark} 

We now wish to apply Theorem \ref{specmaxcor1} to show that $G_\infty \cong C_\infty$ for certain values of $k$, which demands finding a prime ideal dividing $\delta_n$ to odd multiplicity but not dividing any $\delta_m$ for $m \neq n$.  In light of Lemma \ref{relprime}, it is sufficient to know that $v_\p(k) = 0$ to show that $\p$ must divide $(\delta_n)$ for at most one $n$.  So if $\p \mid P_n(1,1)$, it is advantageous to know that $\p \nmid (k)$.  We thus study the divisibility by $k$ of the coefficients of $P_n(X,Y)$ and $Q_n(X,Y)$, given in the following data and Lemma. 

The special case $k =1$ (corresponding to $\phi(x) = (x^2 + 1)/x$) plays an important role.  We thus put 
\begin{align*}
a_n & = P_n(1,1) |_{k=1} & b_n & = Q_n(1,1) |_{k=1}.
\end{align*}  
Note that $a_1 = 2$ and $b_1 = 1$ and 
\begin{align} \label{recurrence}
a_n & = a_{n-1}^2 + b_{n-1}^2 & b_n & = a_{n-1}b_{n-1}
\end{align}
for $n \geq 2$.  We have for instance $a_2 = 5$, $b_2 = 2$, $a_3 = 29$, and $b_3 = 10$.   

Another important quantity is the constant term of $S_n(1,1)$, regarded as a polynomial in $k$.  Set
\begin{align*}
\sigma_n & = S_n(1,1)|_{k=0} & \tau_n & = T_n(1,1)|_{k=0}.
\end{align*} 
We have $\sigma_1 = 2$, $\tau_1 = 1$, and it follows from Lemma \ref{an&bn} that for all $j \geq 1$, 
\begin{align} \label{sigmatau}
			\sigma_{2j+1} &= 
			\sigma_{2j}^2 + \tau_{2j}^2 &
			\sigma_{2j} &= \tau_{2j-1}^2 &
		\tau_j &= \sigma_{j-1}\tau_{j-1}.
\end{align}
For example, we have $\sigma_2 = 1$, $\tau_2 = 2$, $\sigma_3 = 5$, $\tau_3 = 2$, $\sigma_4 = 4$, and  $\tau_4 = 10$.  

The following lemma relates these quantities.  
\begin{lemma} \label{relation}
Let $a_n$ and $\sigma_n$ be defined as above, and let $c$ be the smallest integer that is at least $n/2$.  Then for each $n \geq 2$, $\sigma_n = \prod_{i = 1}^{c} a_i^{e_i}$, with $e_c = 1$ when $n$ is odd
and $e_c = 0$ when $n$ is even.  In particular, $\sigma_n$ is a product of powers of the $a_i$ with $i < n/2 + 1$.  
\end{lemma}

\begin{proof}
The third part of~\eqref{sigmatau} implies that $\tau_j = \sigma_{j-1} \cdots \sigma_1$ for all $j$.  The second part of~\eqref{sigmatau} then gives that 
\begin{equation} \label{evencase}
\sigma_{2j} = (\sigma_{2j-2}\cdots \sigma_1)^2.
\end{equation}
We claim that $\sigma_{2j+1}/\tau_{2j+1} = a_{j+1}/b_{j+1}$.  For $j = 1$ we have
$\sigma_3/\tau_3 =  5/2 = a_2/b_2$, so the claim holds.  Now
$$\frac{\sigma_{2j+1}}{\tau_{2j+1}} = \frac{\sigma_{2j}^2 + \tau_{2j}^2}{\sigma_{2j}\tau_{2j}} = \frac{\tau_{2j-1}^4 + \sigma_{2j-1}^2\tau_{2j-1}^2}{\tau_{2j-1}^3\sigma_{2j-1}} = \frac{\tau_{2j-1}^2 + \sigma_{2j-1}}{\tau_{2j-1}\sigma_{2j-1}} = \frac{1 + (\sigma_{2j-1}/\tau_{2j-1})^2}{\sigma_{2j-1}/\tau_{2j-1}}.$$ 
By inductive assumption the last expression becomes $(1 + a_{j}^2/b_{j}^2)/(a_j/b_j)$, and clearing denominators gives $(b_{j}^2 + a_{j}^2)/(b_{j} a_j)$.  The claim now follows from \eqref{recurrence}.
The claim, together with the fact that $b_{j+1} = a_j \cdots a_1$ and $\tau_{2j+1} = \sigma_{2j} \cdots \sigma_1$, gives
\begin{equation} \label{oddcase}
\sigma_{2j+1}a_j \cdots a_1= \sigma_{2j} \cdots \sigma_1 a_{j+1}.
\end{equation}

We now prove the Lemma by induction.  Since $\sigma_1 = 2 = a_1^1$ and $\sigma_2 = 1 = a_1^0$, the Lemma holds in these cases.  Suppose now that the Lemma holds for all $\sigma_i$ with $i < n$.  If $n = 2j$ for some $j$, then $c = j$ and \eqref{evencase} yields that $\sigma_{2j}$ is a product of powers of the $a_i$.  The maximum index occurring in the right-hand side of \eqref{evencase} is due to $\sigma_{2j-3}$, for which the smallest integer that is at least $(2j-3)/2$ is $j-1$, which is the same as $c-1$.  Hence $\sigma_{2j}$ is a product of powers of the $a_i$ with $i \leq c-1$, as desired.

If $n = 2j+1$, then $c = j+1$.  On the right-hand side of \eqref{oddcase}, we have by inductive hypothesis 
\[
	a_j \mid \sigma_{2j-1}, \quad a_{j-1} \mid \sigma_{2j-3},  \quad \ldots \quad 
	a_1 \mid \sigma_1.
\]
  Hence we may cancel the $a_j \cdots a_1$ on the left-hand side of \eqref{oddcase} to get that $\sigma_{2j+1}$ is a product of powers of the $a_i$.  The largest index occurring on the right-hand side of \eqref{oddcase} is $j+1$, which equals $c$.  Moreover, by inductive hypothesis the largest index appearing in any of the other factors is $j$, showing that $a_{j+1}$ appears to only the first power.  
\end{proof}

\begin{theorem} \label{maxthm}
Let $b = 1$ and assume that $-1$ is not a square in $K$ and each of the fractional $\calO_K$-ideals $(\delta_2), (\delta_3), \ldots, (\delta_m)$ is not the square of a fractional $\calO_K$-ideal.  Assume also that $v_\p(k) = 0$ for all primes $\p$ dividing $a_n := P_n(1,1) |_{k=1}$ for some $n < m/2 + 1$.  Then $G_m \cong C_m$.
\end{theorem}

\begin{proof}
We begin by noting that if $m = 1$ or $m = 2$, then $G_m \cong C_m$ is equivalent to $p_n(x)$ being irreducible, which is ensured by $-1$ and $\pm \delta_2$ not being squares in $K$ (see first paragraph of the proof of Theorem \ref{irredthm}).  

From the proof of Theorem \ref{irredthm} it follows that $-1, \pm \delta_2, \pm \delta_3, \ldots, \pm \delta_m$ not being squares in $K$ implies $p_n(x)$ is irreducible for all $n \leq m$.  We may thus apply Theorem \ref{specmaxcor1}, provided that for each $n$ with $3 \leq n \leq m$ we can find a prime $\p$ with $v_\p(k) = v_\p(2) = 0$, $v_\p(\delta_n)$ odd, and $v_\p(\delta_t) = 0$ for each $t \leq n$.  

Now,
\[
\delta_n = kP_n(1,1) = k^{s(n) + 1}S_n(1,1)
\]
 by Lemma \ref{an&bn}.  Note that $s(n) + 1$ is always even, and thus the squarefree part of the fractional ideal factorization of $(\delta_n)$ (which is non-trivial by hypothesis) divides $S_n(1,1)$.  Therefore there is a prime ideal $\q \subset \calO_K$ with $\q \, || \, (S_n(1,1))$, so that $v_\q(\delta_n) = 1$.  

We first show that $\q \nmid (k)$.  If $\q \mid (k)$, then $0 \equiv S_n(1,1) \equiv S_n(1,1)|_{k=0} \bmod{\q}$.  Thus $\q \mid (\sigma_n)$ in the notation of Lemma \ref{relation}, and that Lemma shows that $\q \mid (a_n)$ for some $n < m/2 +1$, contradicting our hypotheses.  

Because $\q \nmid (k)$, we may apply Lemma \ref{relprime} to show that $\q \nmid (P_i(1,1))$ for all $i \ \neq n$.  It then follows from $kP_i(1,1) = \delta_i$ that $v_\q(\delta_i) = 0$ for all $i \neq n$.  
In particular, since $P_1(1,1) = 2k$, we have that $\q \nmid (2)$, completing the proof.    
\end{proof}

%


\begin{corollary} \label{maincorQ}
Suppose that $K = \Q$.  Then there is a density zero set of primes $\Sigma$, consisting of $2$ and primes congruent to $1$ modulo $4$, such that if $v_p(k) = 0$ for all $p \in \Sigma$, then $G_\infty \cong C_\infty$.
\end{corollary}

\begin{proof} Let $\Sigma$ be the set of primes dividing $a_n := P_n(1,1) |_{k=1}$ for at least one $n \geq 1$.   Note that $a_1 = 2$, and so by assumption $v_2(k) = 0$.  Hence by Theorem \ref{nosquare} and the fact that $\delta_1 = 2k^2$, $\delta_n$ is not a square for all $n \geq 1$.  Because $\delta_n > 0$ for all $n$, this shows that the fractional ideals $(\delta_n)$ are all not squares.  We now apply Theorem \ref{maxthm}, showing that $G_\infty \cong C_\infty$.  

Note that for $k = 1$ it is certainly the case that $v_p(k) = 0$ for all $p \in \Sigma$, and so we have 
$G_\infty \cong C_\infty$ in this case.  However, $a_n$ is the numerator of $\phi^n(1)$ in the case where $k = 1$, and hence Theorem \ref{densitythm} applies to show that the natural density of $\Sigma$ is zero.  

Finally,  a simple induction on the recurrence in~\eqref{recurrence}, shows that $a_n$ and $b_n$ are relatively prime for all~$n$.  Since $a_n$ is the sum of two relatively prime squares, no prime divisors of $a_n$ can be congruent to $3$ modulo 4.  
\end{proof}



\begin{remark} It is easy to see if a given prime $p \in \Z$ belongs to $\Sigma$.  Indeed, letting $\phi(x) = (x^2 + 1)/x$, then we have $\phi^n(1) = a_n/b_n$, provided that $b_n \neq 0$.  Because the preimages of $\infty$ under $\phi$ are $\infty$ and $0$, we have that $b_n = 0$ only when $a_{n-1} = 0$; thus to see if $p \mid a_n$ for some $n$, we need only see if $0$ occurs in the sequence $( \phi^n(1) \bmod{p} )_{n\geq 1}$.  This is easily computable, since the sequence repeats after at most $(p-1)/2$ entries. For instance, the primes in $\Sigma$ less than 2000 are $2, 5, 29, 41, 89, 101, 109, 269, 421, 509, 521, 709$, $929, 941, 1549, 1861$.  Some of these do not divide $a_n$ until $n$ is rather large.  For instance, $929$ divides $a_{42}$, but not $a_n$ for $n < 42$.  
\end{remark}

\begin{remark}
As noted in the Introduction, Corollary \ref{maincorQ} may be far from best possible.  Indeed, we have found no $k \in \Z$ for which $[G_\infty : C_\infty] \geq 2$.  The only $k \in \Q$ where we can be sure this holds are those of the form $a/b$, where $4a^2 + b^2$ is a square, since in that case the numerator of $\phi^2(x)$ factors as two quadratic polynomials (see the remark following Theorem \ref{irredthm}), and note that $kp_2^*(1) = 4a^2 + b^2$, so the image of the action of $\Gal(\overline{\Q}/\Q)$ on the second level of the tree $T_0$ of preimages of zero has order two.  However, the image of $C_\infty$ on the second level of $T_0$ has order four, implying that $[G_\infty : C_\infty] \geq 2$.  
\end{remark}


\section{Density of prime divisors in orbits}

In this section we use the group-theoretic description of $C_n$ given at the beginning of Section \ref{galaut} to show that if $p_n$ is separable and $G_\infty \cong C_\infty$, then for any $a \in K$, the density of the set of primes of $O_K$ dividing some element of the orbit $\{\phi^n(a) : n = 1, 2, \ldots\}$ is zero.  

We begin with a version of \cite[Theorem 2.1]{quaddiv} that applies to a large class of rational functions.  By the {\em natural upper density} of a set of primes $S$ in $O_K$, we mean
\begin{equation} \label{densedef}
D(S) = \limsup_{x \to \infty} \frac{\#\{\p \in S : N(\p) \leq x\}}{\#\{\p : N(\p) \leq x\}},
\end{equation}
where $N(\p) = N_{K/\Q}(\p)$ is the norm of $\p$. 

\begin{theorem} \label{upbound}
Let $\phi \in K(x)$ be a rational function with $p_n$ separable for all $n$, and suppose that 
$\phi^n(\infty) \neq 0$ for all $n > n_0$.  Let $a_n = \phi^n(a_0)$ with $a_0 \in K$.  Then for any $N > n_0$, the density of primes $\p$ of $K$ with $v_\p(a_n) > 0$ for at least one $n \geq 1$ is bounded above by 
\begin{equation} \label{keydense}
\frac{1}{\#G_N} \#\{\sigma \in G_N : \text{$\sigma$ fixes at least one root of $p_N$}\}.
\end{equation}
\end{theorem}

\begin{remark} It is also true that \eqref{keydense} furnishes an upper bound for the density of primes $\p$ such that $0$ is periodic in $O_K/ \p O_K$ under iteration of $\phi$.  This follows from 
the fact that $0$ is periodic in $O_K/ \p O_K$ if and only if $\phi^{-n}(0) \cap O_K$ is non-empty for all $n \geq 1$; cf
\cite[Proposition 3.1]{galmart}.  
\end{remark}

\begin{proof} We denote by $\mathbb{F}_\p$ the field $O_K/ \p O_K$, which is the same as 
$O_{K,\p}/ \p O_{K,\p}$, where $O_{K,\p}$ is the localization of $O_K$ at $\p$.  Any $x \in K$ has a reduction $\overline{x} \in \mathbb{P}^1(\mathbb{F}_\p)$.
Provided that $\p \nmid \Res(p,q)$ (i.e. $\phi$ has {\em good reduction} modulo $\p$), we may reduce the coefficients of $\phi$ modulo $\p$ and obtain a morphism $\overline{\phi} : \mathbb{P}^1(\mathbb{F}_\p)  \to \mathbb{P}^1(\mathbb{F}_\p)$ such that $\overline{\phi}^N = \overline{\phi^N}$ \cite[Theorem 2.15]{jhsdynam}.  

Fix $N > n_0$, and consider 
\[
	\Omega_N 
	= \{\p: \p \nmid \Res(p,q) \; \text{and $\overline{\phi}^N(y) = 0$ has no solution in $	\mathbb{P}^1(\mathbb{F}_\p)$}\}.
\]
If $\p \in \Omega_N$, then we have $\phi^{N + m}(x) = \phi^N(\phi^m(x)) \not\equiv 0 \bmod{\p}$ for all $x \in K$, since otherwise $\overline{\phi}^m(\overline{x}) \in \mathbb{P}^1(\mathbb{F}_\p)$ gives a solution to $\overline{\phi}^N(y) = 0$.  Thus $v_\p(a_{N+m}) \leq 0$ for all $m \geq 0$.  There are only finitely many $\p$ with $v_\p(a_n) > 0$ for some $n < N$, and thus we have 
\begin{equation} \label{krog}
	D(\Omega_N)
	 \leq D\left(\{\p : \text{$v_\p(a_n) \leq 0$ for all $n \geq 1$}\}\right).
\end{equation}

Because $N > n_0$, we have $\phi^N(\infty) \neq 0$, and hence there are only finitely many $\p$ for which 
$\overline{\phi}^N(\infty) = 0$.  But if $\overline{\phi}^N(y) = 0$ has a solution in $\mathbb{P}^1(\mathbb{F}_\p)$, then either $\overline{\phi}^N(\infty) = 0$ or $p_N(x) \equiv 0 \bmod{\p}$ has a solution in $O_K$.  It follows that 
\begin{equation} \label{krog2}
D(\Omega_N) = 1 - D(\{\p : \text{$p_N(x) \equiv 0 \bmod{\p}$ has a solution in $O_K$}\}).
\end{equation}

%

We now use the Chebotarev Density theorem to show that $$D(\{\p : \text{$p_N(x) \equiv 0 \bmod{\p}$ has a solution in $O_K$}\})$$ is given by the expression in \eqref{keydense}, which along with \eqref{krog} and \eqref{krog2} completes the proof.  
Except for finitely many primes ramifying in $K(p_N)$, 
$p_N(x) \equiv 0 \pmod{\p}$ having a solution in $O_K$ is equivalent to $p^N(x)$ having at least one linear factor in $\mathbb{F}_\p[x]$.  Except for possibly finitely many $\p$, this implies that 
$\p O_L = \p_1 \cdots \p_r$, where $L/K$ is obtained by adjoining a root of $p_N$ and at least one of the $\p_i$ has residue class degree one \cite[Theorem 4.12]{narkiewicz}. This is equivalent to the disjoint cycle decomposition of the Frobenius conjugacy class at $\p$ having a fixed point (in the natural permutation representation of $G_N$ acting on the roots of $p_N$).  From the Chebotarev Density Theorem it follows \cite[Proposition 7.15]{narkiewicz} that the density of $\p$ with $\p O_L$ having such a decomposition is the expression in \eqref{keydense}.
\end{proof}

\begin{theorem} \label{densitythm}
Assume the hypotheses of Theorem~$\ref{upbound}$.  Moreover, let $C_n$ be the centralizer in $\Aut(T_n)$ of an involution $\iota \in \Aut(T_n)$ acting non-trivially on $T_1$.  Suppose that $\phi \in K(x)$ satisfies $G_n \cong C_n$ for all $n \geq 1$, and let $a_n = \phi^n(a_0)$ with $a_0 \in K$.  Then 
\begin{equation} \label{krog3}
D(\p \in O_K :  \text{$v_\p(a_n) > 0$ for at least one $n \geq 1$}) = 0.
\end{equation}
\end{theorem}

\begin{proof}
For $n$ large enough, we have from Theorem \ref{upbound} that \eqref{krog3} is bounded above by
\begin{equation} \label{keydense2}
\frac{1}{\#C_n} \#\{\sigma \in C_n : \text{$\sigma$ fixes at least one top-level vertex in $T_n$}\}.
\end{equation}
If $\sigma \in C_n$ satisfies $\sigma|_{T_1} \neq e$, then clearly $\sigma$ can fix no top-level vertices  of $T_n$.  
On the other hand, if $\sigma|_{T_1} = e$ then $\sigma \in \ker(C_n \to C_1)$.  Therefore by Proposition \ref{autprop}, we have that \eqref{keydense2} is the same as
\begin{equation} \label{keydense3}
b_n := \frac{\#\{\sigma \in \Aut(T_{n-1}) : \text{$\sigma$ fixes at least one top-level vertex in $T_{n-1}$}\}}{2\#\Aut(T_{n-1})} .
\end{equation}
From \cite[Propositions 5.5, 5.6]{galmart}, it follows that $b_n = (1 - c_n)/2$, where $c_n$ is given by the evaluation at $z = 0$ of the $n$th iterate of $f(z) = \frac{1}{2}z^2 + \frac{1}{2}$.  This implies that $c_n \rightarrow 1$, and thus $b_n \rightarrow 0$; indeed, it is enough to note that $f$ maps $I = (0, 1]$ to itself, $f(1) = 1$, and $f$ is increasing on $I$.  Moreover, from \cite[Proposition 5.6, part ii]{galmart} we have $b_n = 1/n + O((\log n)/n^2)$.
\end{proof}

\subsection*{Acknowledgments} The authors thank the Institute for Computational and Experimental Research in Mathematics for an enjoyable semester, during which a revision of this paper was completed.  We also thank the referee for numerous helpful comments.

\end{document}